\documentclass[10pt]{article}
\usepackage{lmodern}
\usepackage{amsmath}
\usepackage[T1]{fontenc}
\usepackage[utf8]{inputenc}
\usepackage{authblk}
\usepackage{amsfonts}
\usepackage{graphicx}
\usepackage{rotating}
\usepackage{amssymb}
\usepackage[english]{babel}
\usepackage{color}
\usepackage{amsthm}
\usepackage{graphicx}
\usepackage{tkz-euclide}
\tikzset{elegant/.style={smooth,thick,samples=50,cyan}}
\usepackage{color,tikz}
\usetikzlibrary{patterns}
\usetikzlibrary{decorations.pathreplacing,calc}
\usepackage{rotating}
\usepackage{mathrsfs}
\usepackage{makecell}
\usepackage{microtype}
\usepackage{enumerate}
\usepackage{mathscinet}
\usepackage{array}
\usepackage{multirow}
\usepackage{booktabs}
\usepackage{enumerate}
\usepackage[cal=boondoxo,bb=ams]{mathalfa}
\usepackage{hyperref}
\usepackage{tikz,pgfplots,relsize}
\usepgfplotslibrary{fillbetween}
\usepackage{multicol}

\hypersetup{hidelinks}
\usepackage{titlesec}
\numberwithin{equation}{section}
\newtheorem{defi}{Definition}[section]
\newtheorem{theorem}{Theorem}[section]
\newtheorem{prop}{Proposition}[section]
\newtheorem{lemma}{Lemma}[section]

\newtheorem{remark}{Remark}[section]

\newcommand{\R}{\mathbb{R}}

\newcommand{\N}{\mathbb{N}}

\DeclareMathOperator{\sinc}{sinc\,}
\DeclareMathOperator{\supp}{supp\,}
\DeclareMathOperator{\lin}{lin}
\DeclareMathOperator{\nl}{nl}

\title{$L^p-L^q$ estimates for solutions to the plate equation with mass term}
\author[1]{Alexandre Arias Junior \thanks{Alexandre A. Junior (alexandre.ariasjunior@usp.br)}}
\author[1]{Halit Sevki Aslan \thanks{Halit S. Aslan (halitsevkiaslan@gmail.com)}}
\author[2]{Antonio Lagioia \thanks{Antonio Lagioia (antonio.lagioia@uniba.it)}}
\affil[1]{Department of Computer Science and Mathematics (FFCLRP), University of S\~ao Paulo (USP),
\newline Ribeir\~ao Preto, SP, 14040-901, Brazil}
\affil[2]{Department of Mathematics, University of Bari, Via E. Orabona 4, 70125 Bari, Italy}
\author[1]{Marcelo Rempel Ebert \thanks{Marcelo R. Ebert (ebert@ffclrp.usp.br) (corresponding author)}}
\date{}

\begin{document}

\maketitle

\begin{abstract}
In this paper, we study the Cauchy problem for the linear plate equation with mass term and its applications to semilinear models. For the linear problem we obtain $L^p-L^q$ estimates for the solutions in the full range $1\leq p\leq q\leq \infty$, and we show that such estimates are optimal.
In the sequel, we discuss the global in time existence of solutions to the associated semilinear problem with power nonlinearity $|u|^\alpha$. For low dimension space $n\leq 4$, and assuming $L^1$ regularity on the second datum, we were able to prove global existence for $\alpha> \max\{\alpha_c(n), \tilde\alpha_c(n)\}$ where $\alpha_c = 1+4/n$ and $\tilde \alpha_c = 2+2/n$. However, assuming initial data in $H^2(\mathbb{R}^n)\times L^2(\mathbb{R}^n)$, the presence of the mass term allows us to obtain global in time existence for all $1<\alpha \leq (n+4)/[n-4]_+$. We also show that the latter upper bound is optimal, since we prove that there exist data such that a non-existence result for local weak solutions holds when $\alpha > (n+4)/[n-4]_+$.
\medskip
	
	
\noindent\textbf{AMS Classification (2020)}  35L75; 35L30; 35A01; 35G25.
	
\end{abstract}

\fontsize{12}{15}
\selectfont

\section{Introduction}
\noindent

We consider the forward Cauchy problem for the plate equation with mass
\begin{equation}\label{Eq:MainProblemLineargeneral}
	\begin{cases}
	u_{tt} + \Delta^2u + u = 0, & \,\,\, (t,x)\in (0,\infty)\times \mathbb{R}^n, \\
	u(0,x)=u_0(x), \quad u_t(0,x)=u_1(x), & \,\,\, x\in \mathbb{R}^n.
	\end{cases}
\end{equation}
This  is an important model in the literature, known  as beam operator and plate operator in the case of space dimension $n = 1$ and $n = 2$, respectively. Models to analyse the vibrations of thin plates ($n=2$) given by the full von K\'arm\'an system have been studied by several authors, see~\cite{ciar, lasi2, puel}. We mention that some plate models include also a term~$-\Delta u_{tt}$ called rotational inertia. Energy  estimates for solutions, for which a regularity-loss type decay appears, have been investigated in~\cite{CD09, CDI13p, CHI16, {Ebert20}, SK10}.
 We address the interested reader to \cite{{DAE2017}, {DabbicoEbert2021JMAA}, {Karch}, {Reissig15}} for  $L^p-L^q$ estimates for solutions to structurally
damped plate equation and more general evolution equations.    Strichartz estimates and estimates in modulations spaces for the plate equation were obtained in~\cite{CZ11,CZ11b}.\\
Equations whose ``principal part'' is
$
w_{tt}+(-\Delta)^\sigma w=0,
$
like the plate equation when $\sigma=2$, are called $\sigma$-evolution equations in the sense of Petrowsky, since their symbols $\tau^2 + |\xi|^{2\sigma}$ have only pure imaginary and distinct roots $\tau=\pm i|\xi|^{\sigma}$ for all $\xi \neq 0$. The set of 1-evolution operators coincides with the set of strictly hyperbolic operators. However, several properties of the hyperbolic operators are missing when $\sigma\neq1$.

It is well-known that the Cauchy problem \eqref{Eq:MainProblemLineargeneral} is well-posed in the energy space, namely, if $(u_0, u_1)\in H^2(\mathbb{R}^n)\times L^2(\mathbb{R}^n)$, then there exists a uniquely determined solution   $u \in \mathcal{C}([0, T], H^2(\mathbb{R}^n))\cap \mathcal{C}^1([0, T], L^2(\mathbb{R}^n))$ to \eqref{Eq:MainProblemLineargeneral} for all $T>0$. Moreover, its energy is  conserved (see \cite{EbertReissigBook}), that is,
\[ E(u)(0)=E(u)(t):=\frac12\int_{\mathbb{R}^n} |u_t|^2+ |\Delta u|^2 +|u|^2  dx.\]
As it happens in several equations with oscillatory multipliers, to derive dispersive estimates to \eqref{Eq:MainProblemLineargeneral}, one has to deal with oscillatory integrals. If $u_0\neq 0$, without additional regularity on the initial data, $L^p-L^q$ estimates for solutions  hold only for $p=q=2$.

In \cite{MarshallStrauss1980} the authors considered the Cauchy problem for the Klein-Gordon equation
\[ \begin{cases}
	u_{tt} - \Delta u + u = 0, & \,\,\, (t,x)\in (0,\infty)\times \mathbb{R}^n, \\
	u(0,x)=0, \quad u_t(0,x)=u_1(x), & \,\,\, x\in \mathbb{R}^n,
	\end{cases}
\]
and proved that the operator $T_t:u_1\rightarrow u(t, \cdot)$ is bounded from $L^p$ into $L^q$ if, and only if,
\begin{equation}\label{eq:pq}
n\left(\frac1p-\frac1q\right) + (n-1)\max\left\{\left(\frac12-\frac1p\right),\left(\frac1q-\frac12\right)\right\}\leq 1,
\end{equation}
with the exception for $(p, q)=(1, 2)$ and $(p, q)=(2, \infty)$ if $n=2$. We refer to \cite{{Peral1980}, {Strichartz1970}} for $L^p-L^q$ estimates to the solutions of the Cauchy problem for the free wave equation with the first datum zero.

Our first goal in this paper is to derive $L^p-L^q$ estimates for the solutions to the Cauchy problem for the plate equation with mass term
\begin{equation}\label{Eq:MainProblemLinear}
	\begin{cases}
	u_{tt} + \Delta^2u + u = 0, & \,\,\, (t,x)\in (0,\infty)\times \mathbb{R}^n, \\
	u(0,x)=0, \quad u_t(0,x)=u_1(x), & \,\,\, x\in \mathbb{R}^n.
	\end{cases}
\end{equation}
In the massless case, if $K_{\text{pl}}(t, \cdot)$ denotes the fundamental solution to the free plate equation, then for $1\leq p\leq q\leq\infty$ the operator
\[ K_{\text{pl}}(t,\cdot)\ast: u_1\mapsto u(t,\cdot)=K_{\text{pl}}(t,\cdot)\ast u_1,\]
is bounded from $L^p(\R^n)$ into $L^q(\R^n)$ if, and only if, $d_{\text{pl}}(p,q)\leq 1$ with $p>1$ and $q<\infty$ or if, and only if, $d_{\text{pl}}(p,q)<1$ with $p=1$ and $q=\infty$ (see Theorem 4.1 in ~\cite{Miyachi1981}), where
\begin{equation} \label{Eq:d-plate}
d_{\text{pl}}(p,q)=\frac n2\left(\frac1p-\frac1q\right) + n\max\left\{\left(\frac12-\frac1p\right),\left(\frac1q-\frac12\right)\right\}.
\end{equation}
In this range, by homogeneity, it follows
\[\|u(t,\cdot)\|_{L^q} \lesssim t^{1-\frac{n}{2}\left( \frac{1}{p}-\frac{1}{q} \right)}\|u_1\|_{L^p}, \quad t>0.\]
One of the reasons  that $n$ in $d_{\text{pl}}(p,q)$  replaces $n-1$ in~\eqref{eq:pq} is due to the fact that   the Hessian of $|\xi|^\sigma$, with $\sigma>0$, is singular if, and only if, $\sigma=1$ (see, for instance, \cite{Sjostrand}). Due to this, without difficulties  the obtained results in this paper may be extended to the Cauchy problem for the equation $ u_{tt} + (-\Delta)^\sigma u + u = 0$, for $\sigma>0$, $\sigma\neq 1$.

As in the massless case, the condition $d_{\text{pl}}(p,q)\leq1$ come into play to derive $L^p-L^q$  estimates for solutions at high frequencies in the phase space,  but since for \eqref{Eq:MainProblemLinear} we no longer have the homogeneity property, a challenging problem is to understand the influence of the mass term on the long-time behaviour on the solutions.

A partial answer to this question was obtained 
in the dual line, namely, the following estimate holds for solutions to \eqref{Eq:MainProblemLinear} for all $2\leq q\leq  2^{**} := \frac{2n}{[n-4]_{+}}$  (see \cite{Levandosky1998}):
\begin{equation*}
\|u(t,\cdot)\|_{L^q} \lesssim \|u_1\|_{L^{q^{'}}}\begin{cases}t^{-\frac{n}{4}+\frac{n}{2q}}  & t\geq 1, \\
t^{1-\frac{n}{2}\left( 1-\frac{2}{q} \right)} & t \in (0,1),
\end{cases}
\end{equation*}
where $q'$ denotes the conjugate exponent of $q$. Besides, when $n \geq 5$ the above estimate was extended for $2\leq q <\infty$ (see \cite{Guo2008}).
Using Fourier analysis, stationary phase method and the techniques from WKB analysis (see \cite{DabbicoEbert2021JMAA}, \cite{EbertReissigBook}), in Theorem \ref{Thm:1} we extended  the previous results for $1\leq p\leq q\leq \infty$ out of the dual line, provided that the condition $d_{\text{pl}}(p,q)\leq 1$ is satisfied. Our approach is different from \cite{MarshallStrauss1980} and produce sharp estimates
 to \eqref{Eq:MainProblemLinear}, as it happens for the free plate equation \cite{Ebert19} and to  the structurally
damped plate equation \cite{DabbicoEbert2021JMAA}. However,  it  is not suitable to get sharp estimates to the free wave equation and to the Klein-Gordon equation.  Nevertheless, one may combine the approach used in this paper together with the obtained behaviour for wave type kernels around  the unit sphere (see \cite{Sjostrand}, Lemma 5.1) in order to produce sharp estimates.

The second purpose of this paper is to study the long time behaviour
 of  global (in time) solutions for the semi-linear problem
\begin{equation}\label{Eq:MainProblem}
	\begin{cases}
	u_{tt} + \Delta^2u + u =f(u) , & \,\,\, (t,x)\in (0,\infty)\times \mathbb{R}^n, \\
	u(0,x)=u_0(x), \quad u_t(0,x)=u_1(x), & \,\,\, x\in \mathbb{R}^n,
	\end{cases}
\end{equation}
where $f(u)=|u|^\alpha$, with $\alpha>1$. The local well-posedness for $1<\alpha< \frac{n+4}{[n-4]_+}$ and scattering with small data for $1+\frac8{n}<\alpha< \frac{n+4}{{[n-4]}_+}$ to \eqref{Eq:MainProblem} were studied in the energy space $\mathcal{C}([0, T), H^2(\mathbb{R}^n))\cap  \mathcal{C}^1([0, T), L^2(\mathbb{R}^n))$ in \cite{Levandosky1998}.  We refer to the survey \cite{HebeyPausader} and references therein for further results about  well-posedness, blow-up in
finite time, long time existence, and the existence of uniform bounds for global
solutions to \eqref{Eq:MainProblem}.

 The nonlinearity in \eqref{Eq:MainProblem} may have several shapes, for instance, the derived  results in this paper also hold if  $f(u)=|u|^{\alpha-1}u$ or if
$f$ is locally Lipschitz-continuous satisfying
\[f(0)=0, \quad |f(u)-f(v)| \leq C |u-v|(|f(u)|^{\alpha-1}+|f(v)|^{\alpha-1}),\]
for some $\alpha>1$. Mainly, we are interested in  nonlinearities that can not be included in some energy, so it may  create blow-up in finite time. These kind of non-linearities are known in the literature as of source type.
Semi-linear plate equations with source nonlinearity,  with or without damping,    have been investigated by many authors, see for instance \cite{ DAE2017,  DAE22, Ebert19, {Ebert20}, Reissig15} and references therein. For the semi-linear model \eqref{Eq:MainProblem} with zero  mass, one may have blow-up for small power
nonlinearities, even assuming small data in the energy space.

\medskip

The outline of our study of  the Cauchy problem \eqref{Eq:MainProblem} is the following:
\begin{itemize}
\item We apply the derived estimates in Theorem \ref{Thm:1} to    discuss  the local (in time) existence of Sobolev solutions, even for large data $(u_0, u_1)\in H^2(\mathbb{R}^n)\times L^2(\mathbb{R}^n)$ (see  Theorem \ref{local}).
\item In Theorem \ref{non-existence_local}, under the assumptions $u_0 \equiv 0$ and $u_1\in L^m(\mathbb{R}^n)$, with $m\in [1, 2]$, we prove a non-existence result for local weak solution to \eqref{Eq:MainProblem} when $n > 2m$ and $\alpha > \frac{n+2m}{n-2m}$. 
\item In  Theorems \ref{Thm:GlobalExistence1} and \ref{Thm:GlobalExistence2} the global (in time) existence of solutions are obtained under the assumption of small initial datum $u_1\in L^1(\mathbb{R}^n)$ for  $n\leq 4$ and $\alpha> \max\{\alpha_c(n), \tilde\alpha_c(n)\}$, where such exponents are defined in \eqref{eq:crit exp}.
\end{itemize}

\medskip

\noindent\textbf{Notations:}
\begin{itemize}
\item $f\lesssim g$ means that there exists a positive constant $C$ fulfilling $f\leqslant Cg$, which may be changed in different lines, analogously, for $f\gtrsim g$. Furthermore, the asymptotic relation  $f\simeq  g$ holds when $g\lesssim f\lesssim g$. 

\item $[a]_+$ and $[a]_-$ denote $\max\{a,0\}$ and $\max\{-a,0\}$, respectively. 

\item $p'$ stands for the conjugate exponent of a given $p \in [1,\infty]$.

\item $L_p^q = L_p^q(\mathbb{R}^n)$ denotes the space of tempered distributions $T\in \mathcal{S}'(\R^n)$ such that $T\ast f\in L^q(\R^n)$ for any $f\in \mathcal{S}(\R^n)$, where $\mathcal{S}(\R^n)$ stands for the space of rapidly decreasing Schwartz functions, and
\[ \|T\ast f\|_{L^q} \leq C\|f\|_{L^p} \]
for all $f\in \mathcal{S}(\R^n)$ with a constant $C>0$, which is independent of $f$. In this case, the operator $T\ast$ is extended by density to the space $L^p(\R^n)$.

\item $M_p^q = M_p^q(\mathbb{R}^n)$ stands for the set of Fourier transforms $\hat{T}$ of distributions $T\in L_p^q$, equipped by the norm
\[ \|m\|_{M_p^q} := \sup\big\{ \big\| \mathcal{F}^{-1}\big( m\mathcal{F}(f) \big) \big\|_{L^q}: f\in\mathcal{S}(\R^n), \|f\|_{L^p} = 1 \big\}, \]
and we set $M_p = M_p^p$. The elements in $ M_p^q$ are called multipliers of type $(p,q)$.

\item We fix a nonnegative function $\phi\in \mathcal{C}^\infty(\R^n)$ with compact support in $\{\xi\in\mathbb{R}^n: 2^{-1}\leq |\xi|\leq 2\}$ such that
\[ \sum_{k=-\infty}^{+\infty}\phi_k(\xi) =1, \qquad \text{where} \,\,\, \phi_k(\xi) := \phi(2^{-k}\xi). \]
For any $p\in[1,\infty]$ and $q\in[1,\infty)$ the Besov space is defined as follows:
\[ B_{p,q}^0(\R^n) = \big\{ f\in\mathcal{S}'(\R^n): \forall k\in \mathbb{Z}, \mathcal{F}^{-1}(\phi_k\hat{f})\in L^p(\R^n), \,\,\, \|f\|_{B_{p,q}^0}<\infty \big\}, \]
where
\[ \|f\|_{B_{p,q}^0} = \|\mathcal{F}^{-1}(\phi_k\hat{f})\|_{\ell^q(L^p)} = \left( \sum_{k=-\infty}^{+\infty}\|\mathcal{F}^{-1}(\phi_k\hat{f})\|_{L^p}^q \right)^{\frac{1}{q}}. \]
\end{itemize}



\section{Main results}
\noindent

The main results of this paper read as follows.

\subsection{Estimates for the linear Cauchy problem}
\noindent

\begin{theorem} \label{Thm:1}
Let $u_1\in L^p(\mathbb{R}^n)$. Then, the solution $u$ to the Cauchy problem \eqref{Eq:MainProblemLinear}  satisfies the following estimates:
\begin{equation}\label{linearestimates}
\|u(t,\cdot)\|_{L^q} \lesssim \|u_1\|_{L^p}
\begin{cases}
t^{-\frac{n}{4}\left( \frac{1}{p}-\frac{1}{q} -\beta(p,q) \right)}, & t\geq 1,\\
t^{1-\frac{n}{2}\left( \frac{1}{p}-\frac{1}{q} \right)}, & t \in (0,1),
\end{cases}
\end{equation}
for all $1\leq p\leq q\leq \infty$ such that $d_{\text{pl}}(p,q) < 1$ and for all $1< p\leq 2 \leq  q< \infty$ such that $d_{\text{pl}}(p,q)=1$, where
$d_{\text{pl}}(p,q)$ is given by \eqref{Eq:d-plate} and
\[\beta(p, q):= \begin{cases}
    \quad\left[\frac{1}{p}+\frac{3}{q}-2\right]_{+}, & \text{ if } \frac1p+\frac1q\geq 1, \\
    -\,\left[\frac{3}{p}+\frac{1}{q}-2\right]_{-}, & \text{ if } \frac1p+\frac1q\leq 1,
\end{cases} \]
provided that in the case $(p,q)=(1,3)$ or $(p,q)=(3/2,\infty)$ with $d_{\text{pl}}(p,q) < 1$, the estimate \eqref{linearestimates} is replaced by
\[ \|u(t,\cdot)\|_{L^q}\leq C  t^{-\frac n4 \left(\frac1p-\frac1q\right)} (\log t)^\frac12\|u_1\|_{L^p}\quad  \text{for any} \quad t \geq 1. \]
%
\end{theorem}
\begin{figure}[h]
\centering
  \begin{tikzpicture}[domain=0:1]
   \draw[->] (0,0) -- (5.6,0) node[right] {$1/p$};
   \draw[-] (0,0) -- (0,5.6);
   \node[below left] at (0, 0) {$0$};
   \draw[->] (0,5.6) -- (0,5.6) node[above] {$1/q$};
\path[fill=green!10] (3.7,0.02)--(5,0.02)--(2.5,2.5)--cycle;
\path[fill=green!10] (5,0)--(5,1)--(2.5,2.5)--cycle;
\path[fill=red!30] (1.02,1.01)--(2.5,2.5)--(3.7,0.02)--cycle;
\path[fill=red!30] (1.02,1.02)--(1.7,0.02)--(3.71,0.02)--cycle;
\path[fill=red!10] (2.5,2.5)--(3.7,3.7)--(5,3)--cycle;
\path[fill=red!10] (2.5,2.5)--(5,3)--(5,1)--cycle;
   \draw[color=black, domain=0:5] plot (\x,\x);
   \draw[dashed, color=black]node[left] at (0, 5){\small $1$};
   \draw[-] (3.7,3.7) -- (5,3);
   \draw[->] (4.2,3.5) -- (5.3,3.25);
   \draw[color=black] node[right] at (5.3,3.25) {\small $d_{\text{pl}}(p,q)=1$};
   \draw[-] (2.5,2.5) -- (5,1);
   \draw[->] (4.3,1.5) -- (5.3,1.25);
   \draw[color=black] node[right] at (5.3,1.25) {\small $\frac{1}{p}+\frac{3}{q}-2=0$};
   \draw[-] (2.5,2.5) -- (5,0);
   \draw[dashed, color=black]node[below] at (5, 0){\small $1$};
   \draw[-] (2.5,2.5) -- (3.7,0);
   \draw[dashed, color=black]node[below] at (3.7, 0){$\frac{2}{3}$};
   \draw[->] (3.5,0.8) -- (5.3,0.55);
   \draw[color=black] node[right] at (5.3,0.55) {\small $\frac{3}{p}+\frac{1}{q}-2=0$};
   \draw[-] (1.02,1.01) -- (1.7,0);
   \draw[dashed, color=black]node[below] at (1.7, 0){\small $1-\frac{2}{n}$};
   \draw[->] (1.3,0.5) -- (1,-0.4);
   \draw[color=black] node[below] at (1,-0.4) {\small $d_{\text{pl}}(p,q)=1$};
   \draw[color=black, domain=0:5] plot (\x,5);
   \draw[color=black, domain=0:5] plot (5,\x);
   \draw[dashed, color=black]node[left] at (0, 2.5){$\frac{1}{2}$};
   \draw[dashed](0,2.5)--(5,2.5);
   \draw[dashed, color=black]node[left] at (0, 3){$\frac{2}{n}$};
   \draw[dashed](0,3)--(5,3);
   \draw[dashed, color=black]node[left] at (0, 1){$\frac{1}{3}$};
   \draw[dashed](0,1)--(5,1);
   \draw [draw=black]      (8.5,1.5)   rectangle (12,3.7);
   tikz{\fill[fill=red!10]             (9,3)   rectangle (10,3.3);}
   \draw[color=black] node[right] at (10,3.3) {\small  $t^{-\frac n2+ \frac nq}$};
   tikz{\fill[fill=green!10]            (9,2.5)   rectangle (10,2.8);}
   \draw[color=black] node[right] at (10,2.8) {\small  $t^{-\frac n4\left(\frac1p-\frac1q\right)}$};
   tikz{\fill[fill=red!30]            (9,2)   rectangle (10,2.3);}
   \draw[color=black] node[right] at (10,2.2) {\small  $t^{\frac n2- \frac np}$};
\end{tikzpicture}
\caption{Description of the results in the $1/p-1/q$ plane with the optimal long-time decay estimates}
\end{figure}
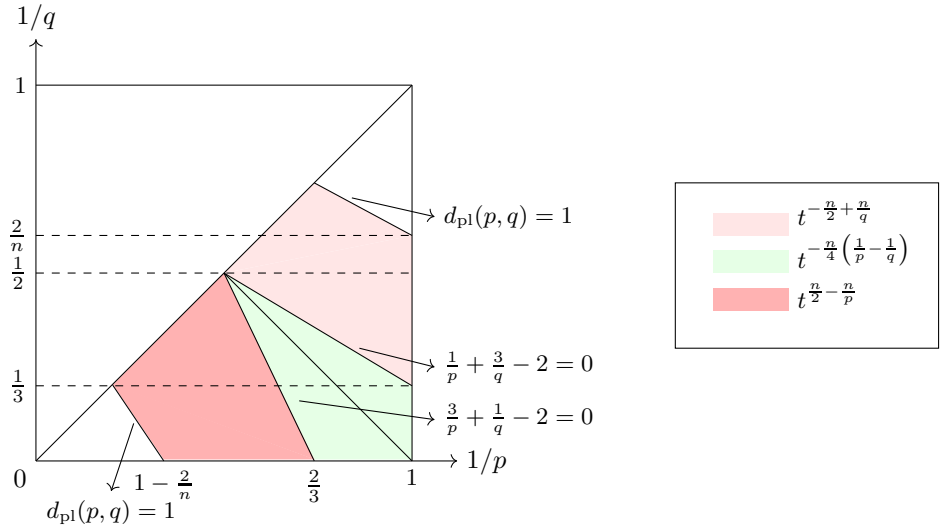
\begin{remark}
We notice that the  estimates in Theorem \ref{Thm:1} are consistent with the obtained ones in \cite{Guo2008, Levandosky1998} on the dual line,
namely, for $p \in [1, 2]$,
\begin{equation*}
\|u(t,\cdot)\|_{L^{p'}} \lesssim \|u_1\|_{L^p}\begin{cases}t^{\frac{n}{4}-\frac{n}{2p}},  & t\geq 1,\\
t^{1-\frac{n}{2}\left( 1-\frac{2}{p'} \right)}, & t \in (0,1).
\end{cases}
\end{equation*}
\end{remark}
\begin{remark}\label{remark_L1_Lq_estimates}
Assuming the initial datum $u_1 \in L^1$, from Theorem \ref{Thm:1} we get the following estimates  for solutions to \eqref{Eq:MainProblemLinear}:
\begin{equation}{\label{eq:L^1-L^p estimates}}
\|u(t,\cdot)\|_{L^q} \lesssim \|u_1\|_{L^1}\begin{cases}t^{-\frac{n}{4}\left(1-\frac{1}{q} \right)}, & q>3,\,t\geq 1,\\
t^{-\frac{n}{2}+\frac{n}{q}}\, (\log t)^\gamma, & q \leq 3,\,t\geq 1,\\
t^{1-\frac{n}{2}\left(1-\frac{1}{q} \right)}\,, & t \in (0,1),
\end{cases}
\end{equation}
provided that $d_{\text{pl}}(1, q)=\frac{n}{2q}<1$, where $\gamma=\gamma(p,q)$ is equal to $1/2$ if $(p,q)=(1,3)$ and zero otherwise.
\end{remark}
\begin{remark}
Let us consider $u_1 \in L^1$. The condition to avoid the singularity when $t\to 0$ in \eqref{eq:L^1-L^p estimates} is 
\[ 1-\frac{n}{2}\left(1-\frac{1}{q} \right)\geq 0,\]
for $q \in [2,\infty]$. The latter condition is satisfied if, and only if,
\begin{equation}\label{eq:nonsing_cond}
    \frac 1q \geq \frac{n-2}n.
\end{equation}
Hence, if $n\leq 2$ and $q \in [1,\infty]$ or
\[ n\geq 3 \quad \text{and} \quad q\leq \frac n{n-2},\]
we get non-singular estimates. In particular, this means that we may work without asking additional regularity for $n\leq 2$ (that is the case of $\alpha_c$), and for $n=3$ when $q\leq 3$. However, if \eqref{eq:nonsing_cond} is not satisfied, we have to ask additional regularity $u_1 \in L^1 \cap L^r$, where $r$ is such that $d_{\text{pl}}(r,q)\leq 1$ and $1-\frac n2\left(\frac 1r - \frac 1q \right)\geq 0$ for any $q \in [2,\infty].$ Choosing $r=2$, we get that $d_{\text{pl}}(2,q)\leq 1$ if and only if $2\leq q\leq \frac{2n}{[n-4]_+}$, with the exception  $q \neq \infty$ if  $n=4$, and the estimates are non-singular for any $q\geq 2$.
\end{remark}

We close this subsection discussing the optimality of the derived estimates in Theorem \ref{Thm:1}. In \cite{Levandosky1998} (see Lemma 2.10), the author proved that
\[ \|K(t,\cdot)\|_{M_p^q} \geq C t^{1-\frac{n}{2}\left( \frac{1}{p}-\frac{1}{q} \right)}, \quad t \in (0,1), \]
and
\[ \|K(t_k, \cdot)\|_{M_p^q} \geq C t_{k}^{-\frac{n}{4}\left( \frac{1}{p}-\frac{1}{q} \right)}, \]
for some sequence $t_k \to +\infty$, where $K$ is defined in \eqref{Eq:kernel-K}. So, it remains only to check the optimality
of
\begin{equation}\label{eq_opti_1}
	\|K(t,\cdot)\|_{M_p^q} \leq t^{-\frac{n}{2}+\frac{n}{q}}, \quad \frac{1}{p}+\frac{1}{q} \geq 1\,\,\,\,\, \text{and} \,\,\,\,\, \frac{1}{p} + \frac{3}{q} \geq 2,
\end{equation}
and
\begin{equation}\label{eq_opti_2}
	\|K(t,\cdot)\|_{M_p^q} \leq t^{\frac{n}{2}-\frac{n}{p}}, \quad \frac{1}{p}+\frac{1}{q} \leq 1\,\,\,\,\, \text{and} \,\,\,\,\, \frac{3}{p} + \frac{1}{q} \geq 2.
\end{equation}
We are going to study only \eqref{eq_opti_1}, because \eqref{eq_opti_2} follows from \eqref{eq_opti_1} by duality. To prove that \eqref{eq_opti_1} is optimal we need the following result, which was inspired by \cite{MarshallStrauss1980}, Lemma 8.
\begin{lemma}\label{Lemma_optimality}
    Let $1 \leq p < \infty$. There exist $f \in L^p$, $b > a > 0$ and $C > 0$ such that there exists $t_k \to \infty$ satisfying
    \begin{equation}\label{eq_estimate_from_below}
    |K(t_k, x) \ast f (x)| \geq C\,t^{-\frac n2}_{k},
    \end{equation}
    for all $x$ such that $at_k < |x| < bt_k$.
\end{lemma}
The proof of Lemma \ref{Lemma_optimality} is quite technical and long, for this reason, we decide to postpone it to an appendix at the end of the paper (see  Appendix \ref{optimality}).\\

Now, take $f \in L^p$, $0 < a < b$ and $t_k$ as in Lemma \ref{Lemma_optimality}. Assume in addition that $\|f\|_{L^p} = 1$. Then we have
    \begin{align*}
        \|K(t_k,x) \ast f(x)\|_{L^q} \geq \left( \int_{at_k<|x|<bt_k} |K(t_k,x) \ast f(x)|^q\,dx \right)^\frac 1q \gtrsim t^{-\frac n2 + \frac nq}_k,
    \end{align*}
   and the optimality of  \eqref{eq_opti_1} is concluded.

\subsection{Existence results to the semilinear Cauchy problem}
\noindent

Based on the Duhamel's principle and the estimates of the linear Cauchy problem from Theorem \ref{Thm:1}, we have the following local existence result.

\begin{theorem}\label{local}(Local Existence)
Let $n\geq 1$, $(u_0, u_1) \in H^2(\mathbb{R}^n)\times L^2(\mathbb{R}^n)$, and suppose that $1< \alpha < \frac{n+4}{[n-4]_+}$. Then, there exists $T >0$ and a local unique Sobolev solution  $u \in \mathcal{C}([0, T], L^q(\mathbb{R}^n))$ to \eqref{Eq:MainProblem}  for all $2\leq q \leq \frac{2n}{[n-4]_+}$, excluding the case $q = \infty$ when $n = 4$.
\end{theorem}

\begin{remark}\label{energysolution}
We point out that Theorem \ref{local} was already obtained in \cite{Levandosky1998} (see Theorem $4.1$). In fact, in \cite{HebeyPausader} (see  Theorem 1.1), applying  Strichartz type estimates,    the authors proved  the local existence of  energy solutions $\mathcal{C}([0, T), H^2(\mathbb{R}^n))\cap  \mathcal{C}^1([0, T), L^2(\mathbb{R}^n))$ for $1< \alpha \leq \frac{n+4}{[n-4]_+}$ to \eqref{Eq:MainProblem}, i.e., they proved that a local existence result is still valid in the endpoint $\alpha = \frac{n+4}{[n-4]_+}$. In this paper we give an alternative (short) proof using the derived estimates to the linear Cauchy problem, but obtaining solution in the space $\mathcal{C}([0, T), L^q(\mathbb{R}^n))$. We also remark that since $ H^2 \subset L^q$ for $2\leq q \leq \frac{2n}{[n-4]_+}$, using Theorem $4.1$ of \cite{Levandosky1998}, we conclude that the solution obtained in Theorem \ref{local} must belong to $\mathcal{C}([0,T), H^2(\R^n))$.
\end{remark}


Under the assumption of small energy data,  one may conclude a priori uniform bounds for the energy solutions to \eqref{Eq:MainProblem} (the same reasoning holds true  for the semilinear Klein-Gordon scalar equation, see \cite{KT}).
Hence, the local solution given by Theorem \ref{local} and Remark \ref{energysolution} can be prolonged in order the get the following result.

\begin{theorem}\label{globalL2}(see, for instance,  Theorem 4.1 in \cite{HebeyPausader})
Let  us  assume  that $1<\alpha\leq \frac{n+4}{[n-4]_+}$. Then, there exists a small constant $\varepsilon>0$ such that for any data $(u_0, u_1) \in H^2(\mathbb{R}^n)\times L^2(\mathbb{R}^n)$ satisfying the assumption
$\|u_0\|_{H^2}+ \|u_1\|_{L^2}< \varepsilon,$
there is a uniquely determined global (in time) energy solution
$\mathcal{C}([0, \infty), H^2(\mathbb{R}^n))\cap  \mathcal{C}^1([0, \infty), L^2(\mathbb{R}^n))$ to \eqref{Eq:MainProblem}.
\end{theorem}

Next, inspired by \cite{ikeda2016remark}, we give a (counterpart of Theorem \ref{local}) non‐existence result for local weak solution to \eqref{Eq:MainProblem} for large values of $\alpha$. Before to state the result we need to recall the definition of weak solution.

\begin{defi}\label{defi_weak_sol}
	We say that $u \in L^{\alpha}_{\text{loc}}([0,T) \times \R^n)$, $T>0$, is a weak solution to \eqref{Eq:MainProblem} with $u_0 \equiv 0$ when
	\begin{equation*}
		\int_0^T \int_{\R^n} u(t,x) \{\partial^2_t + \Delta^2_x + 1\} \psi(t,x) dxdt = \int_0^T \int_{\R^n} |u(t,x)|^{\alpha}  \psi(t,x) dxdt +
		\int_{\R^n} u_1(x) \psi(0,x) dxdt
	\end{equation*}	
	for every $\psi \in \mathcal{C}^{\infty}_{c}([0,T) \times \R^n)$.
\end{defi}

\begin{theorem}\label{non-existence_local}
Let $m \in [1,2]$ and suppose $n > 2m$. If $\alpha > \frac{n+2m}{n-2m}$, then there exists  $0 \neq f \in L^m(\mathbb{R}^n)$ such that if there exists a local  weak solution to \eqref{Eq:MainProblem} with $u_0 \equiv 0$ and $u_1 = \lambda f$, for some $\lambda \geq 0$, then $\lambda = 0$.
\end{theorem}
	
\begin{remark}
Thanks to  Theorems \ref{local}-\ref{globalL2} and Theorem \ref{non-existence_local} with $m=2$, we conclude that for $n\geq 5$, $\alpha=\frac{n+4}{n-4}$ is a critical  exponent for the local (and global) existence of weak solution to \eqref{Eq:MainProblem}.
\end{remark}

\begin{remark}\label{localL1}
If we assume $u_1 \in L^1(\mathbb{R}^n)- L^{1+\epsilon}(\mathbb{R}^n)$ for any $\epsilon>0$ and use the estimates to the linear Cauchy problem in Theorem \ref{Thm:1}, one can still get a local existence result for all $\frac{n}{2}<\alpha < \frac{n}{[n-2]_+}$ to \eqref{Eq:MainProblem} when the space dimension $n\leq 3$ but not for $n\geq 4$. Indeed, in order to get a local unique solution $u \in \mathcal{C}([0, T), L^q(\mathbb{R}^n))$ we have to assume  two compatibility conditions, namely,  $1-\frac{n}{2}\left(1-\frac{1}{q} \right)\geq 0$ in \eqref{eq:L^1-L^p estimates} and $d_{\text{pl}}(1, q)=\frac{n}{2q}<1$, that is,
\[  \frac{n}{2} <q \leq \frac{n}{[n-2]_+}.\]
For this reason, unlike in Theorem \ref{Thm:GlobalExistence1}, in the next Theorem \ref{Thm:GlobalExistence2} we have to assume $u_1 \in L^1(\mathbb{R}^n)\cap L^r(\mathbb{R}^n)$, for some $r\in (1, 2]$.
\end{remark}

In the following results we apply the obtained estimates to the linear Cauchy problem to derive the global (in time) existence of solution to \eqref{Eq:MainProblem} under the assumption of small datum $u_1 \in L^1(\mathbb{R}^n)$.
Since the $L^1 - L^q$ estimates change according whether $q$ is smaller or larger than $3$, we have  to work with the two following candidates to critical exponents
%
\begin{equation} \label{eq:crit exp}
\alpha_c = 1+\frac 4n \qquad \text{and} \qquad \tilde \alpha_c = 2+\frac 2n.
\end{equation}
We point out that $\alpha_c \leq \tilde{\alpha}_c$ if, and only if, $n \leq 2$. Here we use the expression critical exponents in the sense that in the supercritical case, we prove the existence of a global in time solution that satisfies the same optimal decay estimate as the solution of the associated linear Cauchy problem. However, it is not clear what happens in the subcritical case assuming only $u_1 \in L^1(\mathbb{R}^n)- L^{1+\epsilon}(\mathbb{R}^n)$ for any $\epsilon>0$.

\begin{theorem} \label{Thm:GlobalExistence1}
Let $n=1,2$ and assume $u_0 \equiv 0$ and $\alpha>\alpha_c$, where $ \alpha_c$ is given in \eqref{eq:crit exp}. Then, there exists a small constant $\varepsilon>0$ such that for any datum $u_1 \in L^1(\mathbb{R}^n)$ satisfying the assumption
$\|u_1\|_{L^1}< \varepsilon,$
there is a uniquely determined global (in time) Sobolev solution
\begin{align*}
u \in \mathcal{C}([0,\infty), L^2(\mathbb{R}^n) \cap L^\infty(\mathbb{R}^n))
\end{align*}
to \eqref{Eq:MainProblem}. Moreover, the solution satisfies the following decay estimates:
\begin{align*}
\|u(t,\cdot)\|_{L^q} \leq C (1+t)^{-\frac{n}{4}\left( 1-\frac{1}{q} \right)+\frac{n}4 \beta(1, q)}  (\log (1+t))^\gamma\,\|u_1\|_{L^1},
\end{align*}
for all $ q \in [2,\infty]$, where $\gamma$ is defined in \eqref{eq:L^1-L^p estimates}, and the constant $C>0$ does not depend on the initial datum.
\end{theorem}

Due to Remark \ref{localL1}, in the next result we have to ask the additional hypothesis $u_1 \in  L^r(\mathbb{R}^n)$, for some $r\in (1, 2]$. In order to avoid further restrictions on the space of solution and on the upper bound for $\alpha$, we shall assume $u_1 \in L^1(\mathbb{R}^n)\cap  L^2(\mathbb{R}^n)$.

\begin{theorem} \label{Thm:GlobalExistence2}
Let $n=3,4$ and assume $u_0 \equiv 0$ and $\alpha>\tilde{\alpha}_c$, where $\tilde{\alpha}_c$ is defined in \eqref{eq:crit exp}. Then, there exists a small constant $\varepsilon>0$ such that for any $u_1 \in \mathcal{A}:=L^1(\mathbb{R}^n)\cap L^2(\mathbb{R}^n)$ satisfying the assumption
\[ \|u_1\|_{\mathcal{A}}:=\|u_1\|_{L^1}+\|u_1\|_{L^2}<\varepsilon, \]
there exists a uniquely determined global (in time) Sobolev solution
\begin{align*}
u \in \mathcal{C}([0,\infty),  L^2(\mathbb{R}^n) \cap L^{\tilde{p}}(\mathbb{R}^n)),
\end{align*}
for any $\tilde p \in [2,\infty)$ to \eqref{Eq:MainProblem}. Moreover, the following decay estimate holds:
\begin{align*}
&\|u(t,\cdot)\|_{L^q} \leq C (1+t)^{-\frac{n}{4}\left( 1-\frac{1}{q} \right)+\frac{n}4 \beta(1, q)}(\log (1+t))^\gamma \,  \|u_1\|_{\mathcal{A}},
\end{align*}
for all $ q \in [2, \tilde p]$, where $\gamma$ is defined as in \eqref{eq:L^1-L^p estimates}, and the constant $C>0$ does not depend on the initial datum.
\end{theorem}
\begin{remark}
	When $n=3$ we can replace $\tilde{p}$ by $\infty$ in Theorem \ref{Thm:GlobalExistence2}. Moreover, if $u_1 \in L^1(\mathbb{R}^n)- L^{1+\epsilon}(\mathbb{R}^n)$ for any $\epsilon>0$, one can still get  a global (in time) small data Sobolev solution
in $ \mathcal{C}([0,\infty),  L^2(\mathbb{R}^3) \cap L^{3}(\mathbb{R}^3))$ to \eqref{Eq:MainProblem}
 for $\tilde \alpha_c(3)=\frac83<\alpha\leq 3$.
\end{remark}
	
\begin{remark}
	 We get the restriction $n\leq 4$ in Theorem \ref{Thm:GlobalExistence2} because $d_{\text{pl}}(1,\tilde{\alpha}_c)>1$ for $n\geq 5$. We also point out that $d_{\text{pl}}(1,2) = 1$ when $n =4$, so in this situation we need to use $L^r - L^2$ estimates for some $r \in (1,2]$.
\end{remark}

\section{Proof of Theorem \ref{Thm:1}}
\noindent

Let us apply partial Fourier transform with respect to the spatial variable $x$ to the linear problem \eqref{Eq:MainProblemLinear}. Then, $\widehat{u}(t,\xi) = \mathcal{F}_{x\rightarrow\xi}(u(t,x))$ solves
\begin{equation}\label{Eq:FourierTransMainProblem}
	\begin{cases}
	\widehat{u}_{tt} + (1+|\xi|^4)\widehat{u} = 0, & \,\,\, (t,\xi)\in (0,\infty)\times \mathbb{R}^n, \\
	\widehat{u}(0,\xi)=0, \quad \widehat{u}_t(0,\xi)=\widehat{u}_1(\xi), & \,\,\, \xi\in \mathbb{R}^n.
	\end{cases}
\end{equation}
The characteristic roots $\lambda_\pm(|\xi|)$ to the equation in \eqref{Eq:FourierTransMainProblem} are given by
\begin{equation*} \label{Eq:CharasteristicRoots}
\lambda_\pm(|\xi|) = \pm i\sqrt{1+|\xi|^4}\,.
\end{equation*}
The solution to \eqref{Eq:FourierTransMainProblem} is then given by
\begin{equation*} \label{Eq:RepresSolution}
\widehat{u}(t,\xi) = \underbrace{\frac{e^{\lambda_+(|\xi|)t}-e^{\lambda_-(|\xi|)t}}{\lambda_+(|\xi|)-\lambda_-(|\xi|)}}_{{=:\widehat{K}(t,|\xi|)}}\widehat{u}_1(\xi).
\end{equation*}
Our assignment in this part is to process estimates for the Fourier multiplier $\widehat{K}(t,|\xi|)$. Precisely, the kernel can be written by
\begin{equation} \label{Eq:kernel-K}
\widehat{K}(t,|\xi|) = \frac{\sin\big( \sqrt{1+|\xi|^4} \big)t}{\sqrt{1+|\xi|^4}}.
\end{equation}
Then, the solution to \eqref{Eq:MainProblemLinear} can be expressed by
\[ u(t,x) = \mathcal{F}_{\xi\rightarrow x}^{-1}\bigg( \frac{\sin\big( \sqrt{1+|\xi|^4} \big)t}{\sqrt{1+|\xi|^4}} \bigg)\ast_{(x)}u_1(x). \]

Let $\chi\in\mathcal{C}_c^\infty([0,\infty))$ with $\chi(\rho)=1$ for $\rho\in[0,1]$ and $\chi(\rho)=0$ for $\rho\geq 2$. Then, we divide our considerations in two cases, small and large frequencies as follows:
\[\widehat {K}_0(t,\cdot) = \chi(2|\xi|)\widehat {K}(t,\cdot), \qquad\ \widehat {K}_\infty(t,\cdot) = (1-\chi(|\xi|))\widehat {K}(t,\cdot).\]

\subsection{Estimates for high frequencies}
\noindent

Since at high frequencies  one may expect the same behaviour of the free plate equation, we apply the change of variables by $\xi = t^{-\frac{1}{2}}\eta$ in order to get
\[ \|\widehat{K}_\infty(t,\cdot)\|_{M_p^q}= t^{1-\frac n2\left(\frac1p-\frac1q\right)} \|m_\infty(t,\cdot)\|_{M_p^q},\]
where $m_{\infty}$ is the radial multiplier
\begin{equation*} \label{Eq:Multiplier-m}
m_\infty(t,\eta) = \frac{\sin\big( |\eta|^2\sqrt{1+(t/|\eta|^{2})^2} \big)}{|\eta|^2\sqrt{1+(t/|\eta|^2)^2}} \varphi_\infty(t,\eta) = \frac{\sin(\omega_\infty(t,|\eta|) )}{\omega_\infty(t,|\eta|)}  \varphi_\infty(t,\eta)
\end{equation*}
with
\[\omega_\infty(t,|\eta|):=|\eta|^2\sqrt{1+(t/|\eta|^2)^2},\]
and
\[\varphi_\infty(t,\eta)= (1-\chi)(t^{-\frac12}|\eta|).\]
Now, we divide again our considerations in two cases, new small and large frequencies as follows:
\begin{align*}
m_{\infty,0}(t,\eta) &= \chi(|\eta|)m_\infty(t,\eta), \\
m_{\infty,1}(t,\eta) &= \big( 1-\chi(|\eta|) \big)m_\infty(t,\eta).
\end{align*}
We notice that this division is effective only for $t \in (0,1)$, since $m_{\infty,1}(t, \eta) = m_\infty(t,\eta)$ for $t>1$.
\begin{prop} \label{Prop1:LargeFreq}
On the support of $\varphi_\infty(t,\rho)$ the following estimates hold:
\begin{align}
\label{EstimateLargeForOmegaAb} |\partial_\rho^\ell\omega_\infty(t,\rho)| &\leq C_\ell \rho^{2-\ell}, \qquad \ell\in \mathbb N, \\
\label{EstimateLargeForOmegaBe} |\partial_\rho^\ell\omega_\infty(t,\rho)| &\geq c_\ell \rho^{2-\ell}, \,\,\quad\,\,\,\,\,\, \ell=1,2, \\
\label{EstimateLargeForf} \big|\partial_\rho^\ell \frac 1{w_\infty(t,\rho)}\big| &\leq \widetilde{C}_\ell \rho^{-2-\ell}, \quad\,\,\,\, \ell\in \mathbb N .
\end{align}
\end{prop}
\begin{proof}
Observing that
\[\omega_\infty(t,|\eta|)= \sqrt{|\eta|^4+t^2},\]
and $t^2 \leq \rho^4$, we get
\begin{align*}
\partial_\rho\omega_\infty(t,\rho) &= \frac{2\rho^3}{\sqrt{\rho^4+t^2}} \geq \sqrt{2} \rho, \\
\partial_\rho^2\omega_\infty(t,\rho) &= \partial_{\rho} \Big( \frac{2\rho^3}{\sqrt{\rho^4+t^2}} \Big) =
\underbrace{\frac{2\rho^2}{\sqrt{\rho^4+t^2}}}_{\geq \sqrt{2}} \underbrace{\left( 3 - \frac{2\rho^4}{\rho^4+t^2} \right)}_{\geq 1} \geq c \nonumber.
\end{align*}
In the same manner one can prove \eqref{EstimateLargeForOmegaAb} and \eqref{EstimateLargeForf}.
\end{proof}
\begin{lemma} \label{Lemma:Estimate_m0infty}
For any $1\leq p\leq q\leq \infty$ and $t \in [0,1]$, we have
\begin{align*}
\|m_{\infty,0}(t,\cdot)\|_{M_p^q} \leq C.
\end{align*}
\end{lemma}
\begin{proof}
Since $\chi(|\eta|) \in \mathcal{C}_c^\infty$ is independent of $t$, we get (see Theorem $1.8$ in \cite{Hormander_1960_estimates})
\[ \| m_{\infty, 0}(t,\cdot)\|_{M_p^q} \leq C\min\big\{ \|m_{\infty, 0}(t,\cdot)\|_{M_p^p}, \|m_{\infty, 0}(t,\cdot)\|_{M_q^q}  \big\}. \]
Let $(p,q)\notin\{ (1,1), (\infty,\infty), (1,\infty) \}$. We know that $|\sin\rho|\leq C\rho(1+\rho)^{-1}$, that is, $|\sin\rho| \leq \rho$ for small $\rho$ and $|\sin\rho| \leq 1$ for large $\rho$. Moreover, it holds
\[ |\partial_\rho^k\sinc\rho| \lesssim \rho^{-k}, \]
for $\rho \leq 1$, where $\sinc\rho := \rho^{-1}sin \ \rho$. Therefore, using the estimates \eqref{EstimateLargeForOmegaAb}, \eqref{EstimateLargeForf} from Proposition \ref{Prop1:LargeFreq} and chain rule we may conclude that
\[ |\partial_\rho^k\sinc\omega_{\infty}(t,\rho)| \lesssim \rho^{-k}. \]
Thus, by Mikhlin-H\"{o}rmander theorem (cf. Theorem \ref{Sec:App_Thm_Mikhlin-Hormander}) we find that the multiplier norm is bounded by a uniform constant with respect to $t$.
Now we consider $(p,q)=(1,\infty)$. We have
\begin{align*}
\|m_{\infty,0}(t,\cdot)\|_{M_1^\infty} &= \big\| \mathcal{F}_{\eta\to x}^{-1}\big(  \chi(|\eta|)\sinc \omega_\infty(t,|\eta|) \varphi_{\infty}(t,\eta)\big)\big\|_{L^\infty} \\
& \lesssim \big\| \int_{\mathbb{R}^n}e^{ix\cdot\eta}\chi(|\eta|)\sinc \omega_\infty(t,|\eta|) \varphi_{\infty}(t,\eta) d\eta \big\|_{L^\infty} \\
& \lesssim \int_{|\eta|\leq 1}|\sinc \omega_\infty(t,|\eta|)|d\eta\lesssim 1.
\end{align*}
Finally, we consider the case $(p,q)=(1,1)$ (and so by duality also $(p,q)=(\infty,\infty)$). In particular, thanks to $\mathscr{F}(L^1) \subset M_1^1$, it is enough to prove $\mathscr F^{-1} m_{\infty,0} \in L^1$. Since $m_{\infty,0}\in \mathcal{C}_c([0,1] \times \mathbb{R}^n)$, we have that
\[|\partial_\rho^k m_{\infty,0}| \leq C_k, \quad k \in \mathbb{N}.\]
Then, for any integer $N\geq 0$, by $N$ integration by parts
\begin{align*}
\left|\mathscr F^{-1} m_{\infty,0}\,(x) \right|= \left|(2\pi)^{-n} \int_{|\xi|\leq 1} e^{ix\cdot\xi}\,m(\xi)\,d\xi \right| \lesssim |x|^{-N} \int_{|\xi|\leq 1} \,d\xi \lesssim  |x|^{-N}.
\end{align*}
Choosing $N=0$ for $|x|<1$ and $N=n+1$ for $|x|\geq 1$, we can conclude the proof.
\end{proof}
\begin{prop} \label{Prop2:Hessian}
The determinant of the Hessian matrix satisfies the following estimate:
\[ |\det H_{\omega_\infty(t,|\eta|)}| \geq c. \]
\begin{proof}
We follow the proof of Proposition $5.1$ in \cite{DabbicoEbert2021JMAA}. We have
\[ \partial_{\eta_k}\partial_{\eta_j}(\omega_\infty(t,|\eta|)) = \big( \rho^{-1}\partial_\rho\omega_\infty\delta_j^k + \rho^{-1}\partial_\rho(\rho^{-1}\partial_\rho\omega_\infty)\eta_j \eta_k \big)_{\rho=|\eta|}. \]
Then, the Hessian matrix is
\[ H_{\omega_\infty} = \alpha I_n+\beta\eta\otimes\eta/|\eta|^2 \quad \text{with} \quad \alpha = |\eta|^{-1}\partial_\rho\omega_\infty\,, \,\,\, \beta = \partial_\rho^2\omega_\infty-\alpha, \]
where $I_n$ is the identity matrix and $\eta\otimes\eta$ is the matrix with entries $(\eta_k\eta_j)_{j,k}$. Finally, we arrive at
\[ \det H_{\omega_{\infty}} = (1+\beta\alpha^{-1})\alpha^n
=(\partial_\rho^2\omega_{\infty})(|\eta|^{-1}\partial_\rho\omega_{\infty})^{n-1} \geq c,
\]
due to estimates in \eqref{EstimateLargeForOmegaBe}.
\end{proof}
\end{prop}
\begin{lemma} \label{LemmaLargeFreq}
We have the following estimates:
\begin{equation*}
\|m_{\infty,1}(t,\cdot)\|_{M_p^q} \leq C \qquad \text{ for } 1\leq p\leq q\leq \infty \text{ such that  } d_{\text{pl}}(p,q)<1,
\end{equation*}
for any $0<t<1$, and
\begin{equation*}
\|m_{\infty,1}(t,\cdot)\|_{M_p^q} \leq C t^{d_{\text{pl}}-1} \qquad \text{ for } 1\leq p\leq q\leq \infty \text{ such that  } d_{\text{pl}}(p,q)<1,
\end{equation*}
for any $t\geq 1$. Moreover, the estimates remain true for $1< p\leq 2 \leq  q< \infty$ such that $d_{\text{pl}}(p,q)=1$.
\end{lemma}
\begin{proof}
Since $\|\cdot\|_{M_p^q} = \|\cdot\|_{M_{q'}^{p'}}$, it is enough to prove Lemma \ref{LemmaLargeFreq} for $1 \leq p \leq 2$ and $p \leq q\leq p'$.

We are going to use Littman's lemma (cf. Lemma \ref{LittmanLemma}) to estimate the norm
\[ \|m_{\infty,1}(t,\cdot)\phi_k(\cdot)\|_{M_1^\infty} = \big\| \mathcal{F}_{\xi\to x}^{-1}\big[ \big( e^{i\omega_{\infty}(t,|\xi|)}-e^{-i\omega_{\infty}(t,|\xi|)} \big) \phi_k(\xi)f(t,|\xi|) \big] \big\|_{L^\infty}, \]
where
\[
f(t,|\xi|) = \frac{(1-\chi)(|\xi|) \varphi_{\infty}(t,|\xi|)}{\omega_{\infty}(t,|\xi|)}.
\]
%
%
Let $\eta := 2^{-k}\xi$ and $d\eta=2^{-kn}d\xi$. Then, we set $ w_\infty(t,2^k\eta):= 2^{2k}\widetilde w_\infty(t,\eta)$, with $\widetilde w_\infty(t,\eta):= w_\infty(2^{-2k}t,\eta)= \sqrt{t^22^{-4k}+|\eta|^{4}}$. The functions $\widetilde w_\infty$ and $(\widetilde w_\infty)^{-1}$ are uniformly bounded with respect to $t$ and $k$, together with their derivatives on the support of $\phi$. Hence,
\begin{align*}
\big\| \mathcal{F}_{\xi\to x}^{-1}\big[ \big( e^{i\omega_\infty(t,|\xi|)} &- e^{-i\omega_\infty(t,|\xi|)} \big) \phi_k(\xi)\frac 1{ w_\infty(t,\xi)} \big] \big\|_{L^\infty} \\
&= 2^{k(n-2)}\big\| \mathcal{F}_{\eta\to x}^{-1}\big[ e^{i2^{2k}\widetilde w_\infty(t,\eta)} \phi(|\eta|) \frac 1{\widetilde w_\infty(t,\eta)} \big] \big\|_{L^\infty} \\
& \leq C2^{k(n-2)}(1+2^{2k})^{-\frac{n}{2}} \approx 2^{-2k}.
\end{align*}
Thus, Young's inequality implies
\begin{align*}
\big\| \mathcal{F}_{\xi\to x}^{-1}\big( m_{\infty,1}(t,\xi)\phi_k(\xi)\mathcal{F}(u_1) \big) \big\|_{L^\infty} &= \big\| \mathcal{F}_{\xi\to x}^{-1}\big( m_{\infty,1}(t,\xi)\phi_k(\xi) \big)\ast u_1 \big\|_{L^\infty} \\
& \leq  \big\| \mathcal{F}_{\xi\to x}^{-1}\big( m_{\infty,1}(t,\xi)\phi_k(\xi) \big)\big\|_{L^\infty}\|u_1\|_{L^1} \\
& \leq C2^{-2k}\|u_1\|_{L^1},
\end{align*}
which allow us to conlcude that
\begin{equation} \label{Eq:LargeFreqNormM1infty}
\|m_{\infty,1}(\xi)\phi_k(|\xi|)\|_{M_1^\infty} \leq C2^{-2k}.
\end{equation}
Now using the fact that $M_2^2=L^\infty$, we obtain
\begin{align} \label{Eq:LargeFreqNormM22}
\|m_{\infty,1}(t,\cdot)\phi_k\|_{M_2^2} = \|m_{\infty,1}(t,\cdot)\phi_k\|_{L^\infty}\leq C2^{-2k}.
\end{align}
In this way, from \eqref{Eq:LargeFreqNormM1infty}, \eqref{Eq:LargeFreqNormM22} and Riesz-Thorin interpolation theorem (cf. Theorem \ref{Riesz-Thorin}) ($L^2-L^2$ and $L^1-L^\infty$) we get
\begin{equation*} \label{eq:conj_line estimate}
\|m_{\infty,1}(t,\cdot)\phi_k\|_{M_{p_0}^{q_0}}\leq C2^{-2k} \qquad \text{with} \qquad \frac{1}{p_0}+\frac{1}{q_0}=1.
\end{equation*}
Thanks to \eqref{EstimateLargeForOmegaAb}, we have
\[ \|\partial_\xi^\beta(m_{\infty,1}(t,\cdot)\phi_k)\|_{L^2} \leq C_{\beta} \Big( \int_{2^{k-1}\leq|\xi|\leq2^{k+1}}|\xi|^{-4+2|\beta|}d\xi \Big)^{\frac{1}{2}}\leq C_{n,\beta} 2^{k\left( |\beta|-2+\frac{n}{2} \right)}, \]
so, from Bernstein's Theorem (cf. Theorem \ref{Theorem:Bernstein}) we obtain the estimate
\begin{equation} \label{Eq:LargeFreqNormM11E}
\|m_{\infty,1}(t,\cdot)\phi_k\|_{M_1^1} \leq \|m_{\infty,1}(t,\cdot)\phi_k\|_{L^2}^{1-\frac{n}{2N}}\sum_{|\beta|=N}\|\partial_\xi^\beta\big( m_{\infty,1}(t,\cdot)\phi_k \big) \|_{L^2}^{\frac{n}{2N}} \leq C 2^{k(n-2)},
\end{equation}
where we choose $N>n/2$.

Using \eqref{Eq:LargeFreqNormM1infty}, \eqref{Eq:LargeFreqNormM11E} and again Riesz-Thorin interpolation theorem ($L^{p_0}-L^{q_0}$ and $L^1-L^1$), we conclude
\[ \|m_{\infty,1}(t,\cdot)\phi_k\|_{M_{p}^{q}} \leq C2^{k(n\theta-2)} = C2^{k\left( n\left(\frac{1}{p}+\frac{1}{q}-1\right)-2 \right)}, \]
where $0<\theta<1$, $\frac{1}{p}=\frac{1-\theta}{p_0}+\theta$ and $\frac{1}{q}=\frac{1-\theta}{q_0}+\theta$.

Now, we have that
\begin{align*}
    \|m_{\infty,1}(t,\cdot)\|_{M_p^q} \leq \sum_{k=k_1(t)}^\infty \|m_{\infty,1}(t,\cdot) \phi_k\|_{M_p^q},
\end{align*}
where $k_1(t)=\min{\{k \in \mathbb{N}\,:\, 2^{k+1} \geq t^\frac12 \}}.$ We notice that if $0<t<1$, then $k_1(t)=0$ and
\begin{align*}
     \|m_{\infty,1}(t,\cdot)\|_{M_p^q} \leq \sum_{k=0}^\infty \|m_{\infty,1}(t,\cdot) \phi_k\|_{M_p^q} \leq C \sum_{k=0}^\infty 2^{k\left( n\left(\frac{1}{p}+\frac{1}{q}-1\right)-2 \right)},
\end{align*}
which converges if and only if
\[ n\left(\frac{1}{p}+\frac{1}{q}-1\right)-2 <0,\]
or equivalently $d_{\text{pl}}<1.$ Otherwise, if $t \geq 1$, we have
\begin{align*}
     \|m_{\infty,1}(t,\cdot)\|_{M_p^q} \leq &\sum_{k=k_1(t)}^\infty \|m_{\infty,1}(t,\cdot) \phi_k\|_{M_p^q} \\
     & \leq  C\,2^{2(d_{\text{pl}}-1)k_1(t)}\leq C t^{d_{\text{pl}}-1},
\end{align*}
where we used that $d_{\text{pl}}<1$ and $k_1(t) \geq \log_2 t^\frac 12 -1.$ \\

Moreover, we can use embedding for Besov spaces to treat the situation where $d_{\text{pl}}=1$ and $1<p\leq 2\leq q < \infty$. Indeed, in this case we have
\begin{align*}
L^p 	\hookrightarrow B^0_{p, 2}, \quad B^0_{q,2}\hookrightarrow L^q.
\end{align*}
%
%
%
%
%
Then, it holds
\begin{align*}
\| \mathscr{F}^{-1}(m_{\infty,1}(t,\cdot)\widehat{f})\|_{L^q} \leq C_1 \| \mathscr{F}^{-1}(m_{\infty,1}(t,\cdot)\widehat{f})\|_{B^0_{q,2}} \leq C_2 \|f\|_{B^0_{p,2}} \leq C_3 \|f\|_{L^p},
\end{align*}
where the second inequality holds, since
\begin{align*}
 \|\mathscr{F}^{-1}(\phi_k m_{\infty,1}(t,\cdot)\widehat{f})\|_{L^q}^2  \leq C \sum\limits_{j=k-1}^{k+1} \|\mathscr{F}^{-1}(\phi_j \widehat{f})\|_{L^p}^2,
\end{align*}
so that
\begin{align*}
\| \mathscr{F}^{-1}(m_{\infty,1}(t,\cdot)\widehat{f})\|_{B^0_{q,2}}= \left( \sum\limits_{j=0}^{\infty} \|\mathscr{F}^{-1}(\phi_k m_{\infty,1}(t,\cdot)\widehat{f})\|_{L^q}^2 \right)^\frac12 \leq 3C  \left( \sum\limits_{k=0}^{\infty} \|\mathscr{F}^{-1}(\phi_j \widehat{f})\|_{L^p}^2 \right)^\frac12.
\end{align*}
\end{proof}

Summing up, we have proven the following lemma.

\begin{lemma}\label{high freq}
Let $1\leq p\leq q\leq \infty$. Then, we have
\begin{equation*}
 \|\widehat{K}_\infty(t,\cdot)\|_{M_p^q}\leq C \begin{cases} t^{1-\frac n2\left(\frac1p-\frac1q\right)  +d_{\text{pl}}-1}, & \ {\text for \ any } \  t\geq 1,\\
 t^{1-\frac n2\left(\frac1p-\frac1q\right)}, & \ t\in (0,1),
 \end{cases}
\end{equation*}
provided that $d_{\text{pl}}(p,q)<1$. Moreover, these estimates remain true for $1< p\leq 2 \leq  q< \infty$ such that $d_{\text{pl}}(p,q)=1$.
\end{lemma}
%

\subsection{Estimates for small frequencies}

\noindent

It is clear that $\tilde K_0(t, \cdot)\in M_p^q$ for all $t\in [0,1]$ and for all $1\leq p\leq q\leq \infty$. Moreover,
\[\|\widehat{K}_0(t,\cdot)\|_{M_p^q}\leq C t, \quad t\in [0,1].\]
  To estimate $\|\widehat{K}_0(t,\cdot)\|_{M_p^q}$ for $t\geq 1$, we are going to use again the stationary phase method as in the previous subsection. So, in order to get uniformly estimates w.r.t. time for the phase function, we perform the change of variable $\eta=t^\frac14 \xi$. Thus, we obtain
\[ \|\widehat{K}_0(t,\cdot)\|_{M_p^q}= t^{-\frac n4\left(\frac1p-\frac1q\right)} \|m_0(t,\cdot)\|_{M_p^q}, \]
where we consider the radial multiplier
\begin{equation*} \label{Eq:Multiplier-m0}
	m_0(t,\eta) = \frac{\sin\big( t\sqrt{1+ t^{-1}|\eta|^{4}}\big)}{\sqrt{1+t^{-1}|\eta|^4)}} \varphi_0(t,\eta) = \frac{\sin(\omega(t,|\eta|) )}{\sqrt{1+t^{-1}|\eta|^4)}}  \varphi_0(t,\eta),
\end{equation*}
where
\[
\varphi_0(t,\eta)= \chi(t^{-\frac14}|\eta|), \quad \omega(t,\eta)=t \sqrt{1+t^{-1}|\eta|^4}.
\]
Now, we divide again our considerations in two cases, new small and large frequencies as follows:
\begin{align*}
	m_{0,0}(t,\eta) &= \chi(|\eta|)m_0(t,\eta), \\
	m_{0,1}(t,\eta) &= \big( 1-\chi(|\eta|) \big)m_0(t,\eta).
\end{align*}


\begin{prop} \label{Prop1:SmallFreq}
For $t^{-1}\rho^4 <1$ the following estimates hold:
\begin{align}
\label{EstimateSmallForOmegaAb} |\partial_\rho^\ell\omega(t,\rho)| &\leq C_\ell \rho^{4-\ell}, \qquad \ell\in \mathbb N\setminus\{0\}, \\
\label{EstimateSmallForOmegaBe} |\partial_\rho^\ell\omega(t,\rho)| &\geq c_\ell \rho^{4-\ell}, \,\,\qquad \ell=1,2, \\
\label{EstimateSmallForf2} \big|\partial_\rho^\ell \frac t{w(t,\rho)}\big| &\leq \widetilde{C}_\ell\rho^{-\ell}, \quad\,\,\,\,\,\,\,\,\, \ell\in \mathbb N .
\end{align}
\end{prop}
\begin{proof}
We may estimate
\begin{align*}
\partial_\rho\omega(t,\rho) &= \frac{2\rho^3}{\sqrt{t^{-1}\rho^4+1}} \approx \rho^3, \\
\partial_\rho^2\omega(t,\rho) &= \partial_{\rho} \Big( \frac{2\rho^3}{\sqrt{t^{-1}\rho^4+1}} \Big) = \frac{6\rho^2}{\sqrt{t^{-1}\rho^4+1}}-\frac{4\rho^6 t^{-1}}{(t^{-1}\rho^4+1)\sqrt{t^{-1}\rho^4+1}}\approx \rho^2.
\end{align*}
In the same way, we conclude \eqref{EstimateSmallForOmegaAb}.
On the other hand, we have
\[ \big|\partial_\rho^\ell \frac 1{w(t,\rho)}\big|\lesssim \frac{w^{(\ell)}}{w^2} + \frac{(w')^\ell}{w^{\ell+1}} \lesssim t^{-2}\rho^{4-\ell} + t^{-\ell-1}\rho^{3\ell},\]
and so, using $t^{-1}\rho^4 <1$, we obtain \eqref{EstimateSmallForf2} and this concludes the proof.
\end{proof}

\begin{prop}
The determinant of the Hessian matrix satisfies the following estimate:
\[ |\det H_{\omega(t,|\eta|)}| \geq \rho^{2n}. \]
\begin{proof}
Following the proof of Proposition \ref{Prop2:Hessian} and due to estimates \eqref{EstimateSmallForOmegaBe}, we arrive at
\[ \det H_\omega = (1+\beta\alpha^{-1})\alpha^n = (\partial_\rho^2\omega)(|\eta|^{-1}\partial_\rho\omega)^{n-1} \geq \rho^{2n}. \]
\end{proof}\end{prop}

\begin{lemma} \label{Lemma:Estimate_m00}
For any $1\leq p\leq q\leq \infty$, we have
\begin{align*}
\|m_{0,0}(t,\cdot)\|_{M_p^q} \leq C.
\end{align*}
\end{lemma}

\begin{proof}
Since $\chi\in \mathcal{C}_c^\infty$ is independent of $t$, we get (see Theorem $1.8$ in \cite{Hormander_1960_estimates})
\[ \| m_{0,0}(t,\cdot)\|_{M_p^q} \leq C\min\big\{ \|m_{0,0}(t,\cdot)\|_{M_p^p}, \|m_{0,0}(t,\cdot)\|_{M_q^q}  \big\}. \]
Let $(p,q)\notin\{ (1,1), (\infty,\infty), (1,\infty) \}$. Using estimates \eqref{EstimateSmallForOmegaAb} and \eqref{EstimateSmallForf2}, we have
\[ |\partial_\rho^k  \sin w(t,\rho)| \leq  \rho^{4-k},\quad k\geq 1, \]
for $\rho \leq 1$. Then, from Leibniz rule we may conclude that ($\rho < 1$)
\[ |\partial_\rho^k \sin w(t,\rho)\, t/w(t,\rho)| \lesssim \sum_{1\leq \ell \leq k} \rho^{4-\ell}\rho^{-k+\ell} + C \rho^{-k} \lesssim \rho^{-k}. \]
Thus, by Mikhlin-H\"{o}rmander theorem (cf. Theorem \ref{Sec:App_Thm_Mikhlin-Hormander}) we find that the multiplier norm is bounded by a uniform constant with respect to $t$.

Now we consider $(p,q)=(1,\infty)$. Then, we have
\begin{align*}
\|m_{0,0}(t,\cdot)\|_{M_1^\infty} &= \big\| \mathcal{F}_{\eta\to x}^{-1}\big(  \chi(|\eta|)\sin \omega(t,|\eta|)/\sqrt{1+t^{-1}|\eta|^4} \big)\big\|_{L^\infty} \\
& \lesssim \int_{|\eta|\leq 1}|\sin\omega(t,|\eta|)|d\eta\lesssim 1.
\end{align*}
Finally we consider the case $(p,q)=(1,1)$. Using $\rho<1$, $t>1$ and the smoothness of the multiplier, we notice that we may improve estimate \eqref{EstimateSmallForf2}, obtaining
\[\big|\partial_\rho^k \frac t{w(t,\rho)}\big| \leq C_k,\]
so that we may estimate
\[ |\partial_\rho^k m_{0,0}(t,\rho)| \lesssim \tilde{C}_k. \]
So, we may conclude the proof with integration by parts argument, as in the proof of Lemma \ref{Lemma:Estimate_m0infty}.
\end{proof}

\begin{lemma} \label{LemmaSmallFreq1}
We have the following estimates:
\begin{equation*}
\|m_{0,1}(t,\cdot)\|_{M_p^q} \leq C t^{\frac{n}{4}\beta(p,q)} \qquad \text{ for } 1\leq p\leq q\leq \infty,
\end{equation*}
for any $t\geq 1$, where $\beta=\beta(p,q)$ is defined in Theorem \ref{Thm:1}. 
\end{lemma}
\begin{proof}
Since $\|\cdot\|_{M_p^q} = \|\cdot\|_{M_{q'}^{p'}}$ it is enough to prove Lemma \ref{LemmaSmallFreq1} for $1 \leq p \leq 2$ and $p \leq q\leq p'$.

We are going to use Littman's lemma (cf. Lemma \ref{LittmanLemma}) in order to find an estimate to the norm
\[ \|m_{0,1}(t,\cdot)\phi_k\|_{M_1^\infty} = \left\| \mathcal{F}_{\xi\to x}^{-1}\left[ \big( e^{i\omega(t,|\xi|)}-e^{-i\omega(t,|\xi|)} \big) \phi_k(\xi)\frac {\varphi_0(t,\eta)(1-\chi)(\eta)}{\sqrt{1+t^{-1}|\xi|^4}} \right] \right\|_{L^\infty}. \]
%
%

Let $\eta := 2^{-k}\xi$ and $d\eta=2^{-kn}d\xi$.
Then, we set $ w(t,2^k\eta):= 2^{4k}\widetilde w(t,\eta)$, with
$$
\widetilde w(t,\eta) := w(2^{-4k}t,\eta)=\sqrt{t^2 2^{-8k}+t 2^{-4k}|\eta|^{4}}.
$$
Thus, we may conclude
\begin{align*}
\big\| \mathcal{F}_{\xi\to x}^{-1}\big[ \big( e^{i\omega(t,|\xi|)} &- e^{-i\omega(t,|\xi|)} \big) \phi_k(\xi)\frac 1{\sqrt{1+t^{-1}|\xi|^4}} \big] \big\|_{L^\infty} \\
&\lesssim 2^{kn}\big\| \mathcal{F}_{\eta\to x}^{-1}\big[ e^{i2^{4k}\widetilde w(t,\eta)} \phi(|\eta|) \frac 1{\sqrt{1+t^{-1}|2^k\eta|^4}}\big] \big\|_{L^\infty} \\
& \leq C2^{kn}(1+2^{4k})^{-\frac{n}{2}} \approx 2^{-kn},
\end{align*}
due to the fact that the derivatives  of the function $\widetilde w$ are uniformly bounded with respect to $t$ and $k$ on the support of $\phi$ and also the amplitude function $(1+t^{-1}|2^k\eta|^4)^{-1/2}$ is uniformly bounded with respect to t and k, together with all its $\eta-$derivatives on the support of $\phi$.

From Young's inequality we get
\begin{align*}
\big\| \mathcal{F}_{\xi\to x}^{-1}\big( m_{0,1}(t,\xi)\phi_k(\xi)\mathcal{F}(u_1) \big) \big\|_{L^\infty} &= \big\| \mathcal{F}_{\xi\to x}^{-1}\big( m_{0,1}(t,\xi)\phi_k(\xi) \big)\ast u_1 \big\|_{L^\infty} \\
& \leq  \big\| \mathcal{F}_{\xi\to x}^{-1}\big( m_{0,1}(t,\xi)\phi_k(\xi) \big)\big\|_{L^\infty}\|u_1\|_{L^1} \\
& \leq C2^{-kn}\|u_1\|_{L^1},
\end{align*}
which allows to conclude the following:
\begin{equation} \label{Eq:SmallFreqNormM1infty_1}
\|m_{0,1}(\xi)\phi_k(|\xi|)\|_{M_1^\infty} \leq C2^{-kn}.
\end{equation}
Using the fact that $M_2^2=L^\infty$, we obtain
\begin{align} \label{Eq:SmallFreqNormM22_1}
\|m_{0,1}(t,\cdot)\phi_k\|_{M_2^2} = \|m_{0,1}(t,\cdot)\phi_k\|_{L^\infty}\leq C.
\end{align}
Now, using \eqref{Eq:SmallFreqNormM1infty_1}, \eqref{Eq:SmallFreqNormM22_1} and Riesz-Thorin interpolation theorem (cf. Theorem \ref{Riesz-Thorin}) ($L^2-L^2$ and $L^1-L^\infty$) we obtain
\begin{equation*} \label{eq:conj_line estimate2}
\|m_{0,1}(t,\cdot)\phi_k\|_{M_{p_0}^{q_0}}\leq C2^{-kn \big(1-\frac 2q_0\big)} \qquad \text{with} \qquad \frac{1}{p_0}+\frac{1}{q_0}=1.
\end{equation*}
Finally, since we may estimate
%
%
\begin{align*}
     |\partial_\rho^\ell sin(w(t,\rho)) | \lesssim \rho^{3\ell},\\
     |\partial_\rho^\ell m_{0,1}(t,\rho) | \lesssim \rho^{3\ell},
\end{align*}
we have
\begin{equation*}
\|m_{0,1}(t,\cdot)\phi_k\|_{L^2}\lesssim 2^{kn/2}.
\end{equation*}
Hence, from Bernstein's Theorem (cf. Theorem \ref{Theorem:Bernstein}) we get
\begin{equation} \label{Eq:SmallFreqNormM11_1}
\|m_{0,1}(t,\cdot)\phi_k\|_{M_1^1} \leq \|m_{0,1}(t,\cdot)\phi_k\|_{L^2}^{1-\frac{n}{2N}}\sum_{|\beta|=N}\|\partial_\xi^\beta\big( m_{0,1}(t,\cdot)\phi_k \big) \|_{L^2}^{\frac{n}{2N}} \leq C 2^{2kn}.
\end{equation}
Using \eqref{Eq:SmallFreqNormM1infty_1}, \eqref{Eq:SmallFreqNormM11_1} and again Riesz-Thorin interpolation theorem ($L^{p_0}-L^{q_0}$ and $L^1-L^1$), we conclude
\[ \|m_{0,1}(t,\cdot)\phi_k\|_{M_{p}^{q}} \leq C2^{kn\left( \frac{1}{p}+\frac{3}{q}-2 \right)}, \quad \frac{1}{p} + \frac{1}{q} \geq 1. \]
%
Now, we have that
\begin{align*}
    \|m_{0,1}(t,\cdot)\|_{M_p^q} \leq \sum_{k=0}^{k_0(t)}\|m_{0,1}(t,\cdot) \phi_k\|_{M_p^q},
\end{align*}
where $k_0(t)=\max{\{k \in \mathbb{Z}\,:\, 2^{k-1} \leq t^\frac14 \}}.$
Hence, we have
\begin{align*}
     \|m_{0,1}(t,\cdot)\|_{M_p^q} \leq \sum_{k=0}^{k_0(t)}\|m_{0}(t,\cdot) \phi_k\|_{M_p^q} \leq C \sum_{k=0}^{k_0(t)} 2^{kn\left( \frac{1}{p}+\frac{3}{q}-2 \right)},
\end{align*}
which converges (and with an estimate independent of time) if and only if
\[ \frac{1}{p}+\frac{3}{q}-2 <0.\]

In the situation where $\frac{1}{p}+\frac{3}{q}-2 =0$ we may use embeddings for Besov spaces as in the previous section. Let $1<p\leq 2\leq q < \infty$ (the line $\frac{1}{p}+\frac{3}{q}-2 =0$ in the $\frac 1p-\frac1q$ plane is localized in the zone $ 1\leq p\leq 2\leq q\leq 3$). Since
\begin{align*}
L^p 	\hookrightarrow B^0_{p, \max\{p,2\}}, \quad B^0_{q,\min\{2,q\}}\hookrightarrow L^q,
\end{align*}
we have
\begin{align*}
\| \mathscr{F}^{-1}(m_{0,1}(t,\cdot)\widehat{f})\|_{L^q} \leq C_1 \| \mathscr{F}^{-1}(m_{0,1}(t,\cdot)\widehat{f})\|_{B^0_{q,2}} \leq C_2 \|f\|_{B^0_{p,2}} \leq C_3 \|f\|_{L^p},
\end{align*}
where the second inequality holds, since
\begin{align*}
 \|\mathscr{F}^{-1}(\psi_k m_{0,1}(t,\cdot)\widehat{f})\|_{L^q}^2  \leq C \sum\limits_{j=k-1}^{k+1} \|\mathscr{F}^{-1}(\psi_j \widehat{f})\|_{L^p}^2,
\end{align*}
so that
\begin{align*}
\| \mathscr{F}^{-1}(m_{0,1}(t,\cdot)\widehat{f})\|_{B^0_{q,2}}= \left( \sum\limits_{k=0}^{k_0} \|\mathscr{F}^{-1}(\psi_k m_{0,1}(t,\cdot)\widehat{f})\|_{L^q}^2 \right)^\frac12 \leq 3C  \left( \sum\limits_{k=0}^{k_0} \|\mathscr{F}^{-1}(\psi_k \widehat{f})\|_{L^p}^2 \right)^\frac12.
\end{align*}
In the case~$p=1$ and~$q=3$, using only the second embedding, we obtain
\begin{align*}
&\| \mathscr{F}^{-1}(m_{0,1}(t,\cdot)\widehat{f})\|_{L^q} \leq C_1 \| \mathscr{F}^{-1}(m_{0,1}(t,\cdot)\widehat{f})\|_{B^0_{q,\min\{2,3\}}} \leq C_2 \left( \sum\limits_{k=0}^{k_0} \|\mathscr{F}^{-1}(\psi_k \widehat{f})\|_{L^1}^{2} \right)^\frac1{2}\\
&\leq C_3  \|f\|_{L^1}\left( \sum\limits_{k=0}^{k_0} 1 \right)^\frac1{2} \leq  C_4 (\log t)^\frac1{2}\|f\|_{L^1}.
\end{align*}
Finally, we assume $\beta(p,q) = \frac{1}{p}+\frac{3}{q}-2 >0$. In this case, we may estimate
\begin{align*}
\|m_{0,1}(t,\cdot)\|_{M_p^q}\leq \frac{2^{\{k_0(t)+1\}n\beta(p,q)}-1}{2^{n\beta(p,q)}-1} \lesssim 2^{k_0(t)n\beta(p,q)} \approx t^{\frac{n\beta(p,q)}{4}}.
\end{align*}
%
\end{proof}
Summing up, we have proven the following result.
\begin{lemma}\label{low freq}
Let $1\leq p\leq q\leq \infty$. Then, setting $\beta(p,q)$ as in Theorem \ref{Thm:1}, we have
\begin{equation*}
\|\widehat{K}_0(t,\cdot)\|_{M_p^q}\leq C \begin{cases} t^{-\frac n4\left(\frac1p-\frac1q\right)+\beta(p,q)\frac{n}{4}}, & \ {\text for \ any } \  t\geq 1,\\
t, & \ t\in [0,1),
 \end{cases}
\end{equation*}
provided that in the case $(p,q)=(1,3)$ or $(p,q)=(3/2,\infty)$ with $d_{\text{pl}}< 1$, estimate \eqref{low freq} is replaced by
\[ \|\widehat{K}_0(t,\cdot)\|_{M_p^q}\leq C  t^{-\frac n4 \left(\frac1p-\frac1q\right)} (\log t)^\frac12,\]
for any $t \geq 1$.
\end{lemma}


\begin{proof}[Proof of Theorem \ref{Thm:1}]
Due to
\[ 1-\frac n2\left(\frac1p-\frac1q\right)  +d_{\text{pl}}-1\leq -\frac n4\left(\frac1p-\frac1q\right) + \beta(p,q)\frac{n}{4}, \]
the proof of  Theorem \ref{Thm:1} follows by Lemma \ref{high freq} and  Lemma \ref{low freq}.
\end{proof}


\section{Proof of local and global in time existence results}
\noindent

Let us first introduce some notations which will be used in the proof of Theorems \ref{local}, \ref{Thm:GlobalExistence1} and \ref{Thm:GlobalExistence2}. Let $u^{\lin}(t,x)$ be the solution to \eqref{Eq:MainProblemLineargeneral} and $K=K(t,x)$ be the fundamental solution to the linear Cauchy problem \eqref{Eq:MainProblemLinear}.
Let now $N$ be the operator
\[ N: u\in X(T) \rightarrow Nu(t,x) =  u^{\lin}(t,x)+u^{\nl}(t,x), \]
where $X(T)$ is a suitable space where we search for solutions and 
\begin{align*}
	u^{\nl}(t,x) &= \int_0^tK(t-\tau,x)\ast_{(x)}|u(\tau,x)|^\alpha d\tau.
\end{align*}
Thanks to Duhamel's principle, the solution of \eqref{Eq:MainProblem} is given by the fixed point of the operator $N$. To prove that $N$ admits a fixed point, we are going to show that the mapping $N$ satisfies the following estimates:
	\begin{align}
		\label{Eq:inequality-Nu} \|Nu\|_{X(T)} &\leq C_0\|(u_0, u_1)\|_{\mathcal{A}} + C_1(t)\|u\|_{X(T)}^{\alpha}, \\
		\label{Eq:inequality-NuNv} \|Nu-Nv\|_{X(T)} &\leq C_2(t)\|u-v\|_{X(T)}\Big( \|u\|_{X(T)}^{\alpha-1}  + \|v\|_{X(T)}^{\alpha-1} \Big),
	\end{align}
where $\mathcal{A}$ stands for the space which the data $(u_0, u_1)$ belong to. In this way, by using contraction argument, we get simultaneously a unique solution to $Nu=u$ locally in time for large data and globally in time for small data. To prove the local in time existence we use that $C_1(t)$ and $C_2(t)$ tend to zero as $t$ tends to zero, whereas to prove global in time existence we use that $\max\{C_1(t), C_2(t)\}\leq C$ for all $t\in [0, \infty)$ for a suitable nonnegative constant $C$.

\medskip

\begin{proof}[Proof of Theorem \ref{local}]
		Let $r=\frac{2n}{[n-4]_+}$. For $n \neq 4$ we define the Banach space
		\[ X(T) := \bigl\{ u\in \mathcal{C}([0, T], L^2(\mathbb{R}^n)\cap L^r(\mathbb{R}^n) ) \, : \ \|u\|_{X(T)} := \sup_{t\in[0, T], q\in [2, r]}\|u(t,\cdot)\|_{L^q}<\infty \bigr\}\,, \]
		and for $n=4$ we consider
		\[ X(T) := \bigl\{ u\in \mathcal{C}([0, T], L^2(\mathbb{R}^n)\cap L^r(\mathbb{R}^n) ) \, : \ \|u\|_{X(T)}:=\sup_{t\in[0, T], q\in [2, r)}\|u(t,\cdot)\|_{L^q}<\infty \bigr\}\,. \]
		Due to the hypothesis $(u_0, u_1) \in H^2(\mathbb{R}^n)\times L^2(\mathbb{R}^n)$, it follows from Gagliardo-Nirenberg inequality (see Proposition \ref{GagliardoNirenberginequality}) that
		\[\|u^{\lin}(t,\cdot)\|_{L^q} \lesssim \|u^{\lin}(t,\cdot)\|_{\dot{H^2}}^{\theta}\|u^{\lin}(t,\cdot)\|_{L^2}^{1-\theta}\lesssim E(u^{\lin})(t)=E(u^{\lin})(0)\leq \|u_0\|_{H^2} + \|u_1\|_{L^2},\]
for all $t\geq 0$ and $q \in [2,r]$ when $n \neq 4$ and for all $q \in [2,r)$ when $n =4$.
		 Here $\theta=\frac{n}{2}\left(\frac{1}2-\frac{1}{q}\right)\in [0, 1]$.
		Therefore  $u^{\lin} \in X(T)$.
		
		In the following argument we have to exclude again the case $q = r$ when $n = 4$. Let us fix $\alpha \in \left(1, \frac{n+4}{[n-4]_{+}}\right)$. We then note that for any $q \in [2,r]$ there exists $p = p(\alpha, q) \in \left(\frac{2n}{n+4}, 2\right]$ such that
		$$
		2 \leq p\alpha \leq r \quad \text{and} \quad d_{\text{pl}}(p,q) < 1.
		$$
		So, using Minkowski integral inequality and again the estimates of Theorem \ref{Thm:1} we get
		\begin{align*}
		\|u^{\nl}(t,\cdot)\|_{L^{q}} &\leq \int_0^{t} \|K(t-s,\cdot)\ast_{(x)} |u(s,\cdot)|^{\alpha}\|_{L^{q}}\,ds \\
		&\lesssim \int_0^t (t-s)^{1-\frac{n}{2}\left( \frac{1}{p}-\frac{1}{q} \right)} \|u(s,\cdot)\|^{\alpha}_{ L^{p\alpha}}\,ds\\
		& \lesssim \|u\|_{X(T)}^\alpha \int_0^t (t-s)^{1-\frac{n}{2}\left( \frac{1}{p}-\frac{1}{q} \right)}\,ds \lesssim t^{2-\frac{n}{2}\left( \frac{1}{p}-\frac{1}{q} \right)}\|u\|^{\alpha}_{X(T)},
		\end{align*}
		for any $q \in [2,r]$. Since $p > \frac{2n}{n+4}$ we get
		$$
		2-\frac{n}{2p}+\frac{n}{2q} > 0, \quad \forall q \in [2,r].
		$$
		Therefore, $N$ maps $X(T)$ into itself and, if $T$ is sufficiently small, then the existence of a unique local solution follows by a contraction argument (for
more details, see the proof of Theorem \ref{Thm:GlobalExistence1}).
\end{proof}

In order to prove the global (in time) existence results with data $u_0 \equiv 0$ and $u_1 \in L^1$ the following lemma comes into play.
\begin{lemma}\label{Lem:Auxiliary-GlobalExistence}
	Let $\nu >-1$ and $\mu \in \mathbb{R}$. Then, it holds
	\begin{equation*}
		\int_{0}^{t}(t-s)^{\nu}\,(1+s)^{\mu}\,ds\lesssim
		\begin{cases}
			(1+t)^{\nu}, &   \mbox{ if }\,\, \mu <-1,\\
			(1+t)^{\nu}\,log(e+t), &  \mbox{ if }\,\, \mu =-1,\\
			(1+t)^{1+\nu+\mu}, &  \mbox{ if }\,\,  \mu >-1,
		\end{cases}
	\end{equation*}
	and
	\[\int_{0}^{t}(t-s)^{\nu}\,e^{-c(t-s)}(1+s)^{\mu}\,ds\lesssim (1+t)^\mu .\]
	Moreover, the estimate is also valid if $(t-s)^{\nu}$ is replaced by $(1+t-s)^{\nu}$ in the integral.
\end{lemma}
\begin{proof}[Proof of Theorem \ref{Thm:GlobalExistence1}]
	We introduce the solution space $X(T)$ by
	\begin{align*}
		X(T):=\big\{ u\in \mathcal{C}\big( [0,T],L^2(\R^n)\cap L^\infty(\R^n) \big): \|u\|_{X(T)}<\infty \big\},
	\end{align*}
	where
	\begin{align*}
		\|u\|_{X(T)}&:=
		\sup\limits_{t \in [0,T],\, q \in [2,3)} (1+t)^{\frac{n}{4}\left(1-\frac{1}{q} \right)-\frac{n}{4} \beta(1, q)} \|u\|_{L^q}
		+ (1+t)^{\frac{n}{4}\left(1-\frac{1}{3} \right)} \, (\log (1+t))^{-\frac12}\, \|u\|_{L^3} \\
		&\quad +\sup\limits_{t \in [0,T], q \in (3,\infty]} (1+t)^{\frac{n}{4}\left(1-\frac{1}{q} \right)} \|u\|_{L^q}.
	\end{align*}
	Since $u_1 \in L^1(\R^n) = \mathcal{A}$, from \eqref{eq:L^1-L^p estimates} we immediately conclude
	\[ \|u^{\lin}\|_{X(T)} = \|K(t,x)\ast u_1(x)\|_{X(T)} \leq C_0\|u_1\|_{\mathcal{A}}. \]
	%
	%
	%
	Therefore, it remains to prove
	\[ \|u^{\nl}\|_{X(T)} \lesssim \|u\|_{X(T)}^\alpha.\]
	For the sake of brevity we shall consider only the case $q \neq 3$.
	Using the estimates from \eqref{eq:L^1-L^p estimates}, we have the following:
	%
	\begin{align*}
		\|u^{\nl}(t,\cdot)\|_{L^q} &\lesssim \int_0^t\big\| K(t-\tau,x)\ast_{(x)}|u(\tau,x)|^\alpha \big\|_{L^q}d\tau \\
		& \lesssim \int_0^t(1+t-\tau)^{-\frac{n}{4}\left( 1-\frac{1}{q} \right)+\frac{n}4 \beta(1, q)}\,\||u(\tau,x)|^\alpha \|_{L^1} d\tau \\
		& \lesssim \int_0^t(1+t-\tau)^{-\frac{n}{4}\left( 1-\frac{1}{q} \right)+\frac{n}4 \beta(1, q)}\|u\|_{L^\alpha}^\alpha d\tau \\
		& \lesssim \|u\|_{X(T)}^\alpha\int_0^t(1+t-\tau)^{-\frac{n}{4}\left(1-\frac{1}{q} \right)+\frac{n}4 \beta(1, q)}(1+\tau)^{-\alpha\left( \frac{n}{4}(1-\frac{1}{\alpha})+\frac{n}4 \beta(1, \alpha) \right)}d\tau.
	\end{align*}
	We notice that for $n=1,2$ it holds
	\[ -\frac{n}{4}\Big( 1- \frac{1}{q} \Big)+\frac{n}4 \beta(1, q)>-1. \]
	Moreover, since $\alpha>1+ \frac 4n > 3$ ( for $n =1, 2$) we have $\beta(1,\alpha) = 0$ and therefore
	$$
	\alpha>1+ \frac 4n \iff - \frac{n}{4}\left(\alpha-1\right)<-1.
	$$
 	So, we can use Lemma \ref{Lem:Auxiliary-GlobalExistence} to get
	\begin{align*}
		\|u^{\nl}(t,\cdot)\|_{L^q} \lesssim (1+t)^{-\frac{n}{4}\left( 1-\frac{n}{q} \right)+\frac{n}4 \beta(1, q)}\|u\|_{X(T)}^\alpha.
	\end{align*}

	Next we turn to the proof of the Lipschitz condition \eqref{Eq:inequality-NuNv}.
	Due to the fact that
	\[ |\,|u(\tau,x)|^\alpha - |v(\tau,x)|^\alpha| \lesssim |u(s,x)-v(s,x)|\big( |u(s,x)|^{\alpha-1}+|v(s,x)|^{\alpha-1} \big), \]
	from H\"{o}lder's inequality we obtain
	\begin{align*}
		\big\| |u(s,\cdot)|^\alpha-|v(s,\cdot)|^\alpha \big\|_{L^1} &\lesssim \| u(s,\cdot)-v(s,\cdot) \|_{L^{\alpha}} \big( \|u(s,\cdot)\|_{L^{\alpha}}^{\alpha-1}+\|v(s,\cdot)\|_{L^{\alpha}}^{\alpha-1} \big).
	\end{align*}
	Thus, following the same ideas to what we did to estimate  $\|u^{\nl}(t,\cdot)\|_{L^q}$, we can obtain \eqref{Eq:inequality-NuNv}. In this way the proof of Theorem \ref{Thm:GlobalExistence1} is completed.
\end{proof}

\begin{proof}[Proof of Theorem \ref{Thm:GlobalExistence2}.]
	We introduce the solution space $X(T)$ by
	\begin{align*}
		X(T):=\big\{ u\in \mathcal{C}\big( [0,T],L^2(\R^n)\cap L^{\tilde{p}}(\R^n) \big): \|u\|_{X(T)}<\infty \big\},
	\end{align*}
	where we set $\tilde p=\infty$ for $n=3$ and $\tilde{p} < \infty$ for $n=4$, and the corresponding norm
\begin{align*}
	\|u\|_{X(T)}&:=
	\sup\limits_{t \in [0,T], q \in [2,3)} (1+t)^{\frac{n}{4}\left(1-\frac{1}{q} \right)-\frac{n}{4} \beta(1, q)} \|u\|_{L^q}
	+ (1+t)^{\frac{n}{4}\left(1-\frac{1}{3} \right)} \, (\log (1+t))^{-\frac12}\, \|u\|_{L^3} \\
	&\quad +\sup\limits_{t \in [0,T], q \in (3,\tilde{p}]} (1+t)^{\frac{n}{4}\left(1-\frac{1}{q} \right)} \|u\|_{L^q}.
\end{align*}
	On one hand we have $d(1,q)=\frac n {2q}<1$ for $n=3$ or $n=4$ with $q>2$, and on the other hand $d(2,q)= \frac n4- \frac n{2q} <1$ when $n=3$ or $n=4$ with $2 \leq q< \infty$. Hence, we get
	\begin{align*}
		\|u^{\lin}\|_{X(T)} 
		\lesssim \|u_1\|_{L^1\cap L^2}.
	\end{align*}

	Now we deal with the nonlinear part. For the sake of brevity, we consider only the case $q \neq 3$. Using $L^1\cap L^2-L^q$ estimates for $q \in (2,\tilde p]$, we get
	\begin{align*}
		\|u^{\nl}(t,\cdot)\|_{L^q} &\lesssim \int_0^t\big\| K(t-\tau,x)\ast|u(\tau,\cdot)|^\alpha \big\|_{L^q}\,d\tau \\
		& \lesssim \int_0^t(1+t-\tau)^{-\frac{n}{4}\left( 1-\frac{1}{q} \right)+\frac{n}{4}\beta(1,q)}\,\|\,|u(\tau,\cdot)|^\alpha \|_{L^1 \cap L^2} d\tau \\
		& \lesssim \int_0^t(1+t-\tau)^{-\frac{n}{4}\left( 1-\frac{1}{q} \right)+\frac{n}{4}\beta(1,q)}\big( \|u(\tau,\cdot)\|_{L^\alpha}^\alpha+ \|u(\tau,\cdot)\|_{L^{2\alpha}}^\alpha \big)d\tau.
	\end{align*}
	We note that in both cases, $q \leq 3$ or $q > 3$, we have
	\[ -\frac{n}{4}\left( 1-\frac{1}{q} \right)+\frac{n}{4}\beta(1,q)>-1 \]
	and
	\[ \|u(\tau,\cdot)\|_{L^{\alpha}}^\alpha \lesssim (1+\tau)^{-\sigma}(\log(1+\tau))^{\gamma}\|u\|_{X(T)}^\alpha, \]
	with
	\begin{equation}\label{sigma}\sigma= \begin{cases} \frac{n}{4}\left( \alpha -1 \right) & \alpha>3, \\
	n\left( \frac{\alpha}{2}-1 \right) & \alpha\leq 3.
	\end{cases}
	\end{equation}
	Moreover, we can see that $\|u(\tau,\cdot)\|_{L^{2\alpha}}^\alpha$ decays faster than $\|u(\tau,\cdot)\|_{L^\alpha}^\alpha$, since
	\[-\frac n4 \left(1-\frac 1 {2\alpha} \right) \leq -\frac n2 + \frac n\alpha\]
	if $\alpha \leq 3$ and by definition of $u\in X(T)$ for $\alpha\geq 3$. Thus, using $\alpha > 2+\frac{2}{n}$ we conclude that $\sigma>1$, with $\sigma$ given by \eqref{sigma} and Lemma \ref{Lem:Auxiliary-GlobalExistence} implies
	\begin{align*}
	\|u^{\nl}(t,\cdot)\|_{L^q} \lesssim (1+t)^{-\frac{n}{4}\left( 1-\frac{n}{q} \right)+\frac{\beta}{4}}\|u\|_{X(T)}^\alpha,
	\end{align*}
	for $q \in (2,\tilde p]$. Finally, we consider the case $q=2$. We have
	\begin{align*}
	\|u^{\nl}(t,\cdot)\|_{L^2} &\lesssim \int_0^t\big\| K(t-\tau,x)\ast|u(\tau,\cdot)|^\alpha \big\|_{L^2}\,d\tau \\
	& \lesssim \int_0^t\,\|\,|u(\tau,\cdot)|^\alpha \|_{L^2} d\tau \lesssim \int_0^t \|u(\tau,\cdot)\|_{L^{2\alpha}}^\alpha \,d\tau \lesssim \|u\|_{X(T)}^\alpha\,\int_0^t(1+\tau)^{-\frac n4 \left(1-\frac 1 {2\alpha} \right)\alpha} \,d\tau.
	\end{align*}
	Now, since $\alpha > \tilde \alpha_c = 2+\frac{2}{n}$, we get
	\[-\frac n4 \left(1-\frac 1 {2\alpha} \right)\alpha<-1\]
	and so, using Lemma \ref{Lem:Auxiliary-GlobalExistence}, we conclude
	\[\|u^{\nl}(t,\cdot)\|_{L^2}\lesssim \|u\|_{X(T)}^\alpha.\]
	Following as in the proof of Theorem \ref{Thm:GlobalExistence1}, one may conclude the Lipschitz condition \eqref{Eq:inequality-NuNv}  and this concludes the proof.
	
		
	\end{proof}


\section{Non-local existence result}
\noindent

This section is devoted to the proof of Theorem \ref{non-existence_local}. Let $\eta(t) \in \mathcal{C}^{\infty}_{c}([0,\infty))$ and $\phi(x) \in \mathcal{C}^{\infty}_{c}(\R^n)$ such that $0 \leq \eta, \phi \leq 1$ and
$$
\eta(t)=
\begin{cases}
	1, \quad 0 \leq t \leq \frac{1}{2}, \\
	0, \quad t \geq 1,
\end{cases}
\quad
\phi(x) =
\begin{cases}
	1, \quad |x| \leq \frac{1}{2}, \\
	0, \quad |x| \geq 1.
\end{cases}
$$
We remark that for any $r > 1$ there exists $C_r > 0$ such that
$$
|\partial^{2}_{t} \eta(t)| \leq C_r \eta(t)^{\frac{1}{r}}, \quad |\Delta^{2}_{x} \phi(x)| \leq C_r \phi(x)^{\frac{1}{r}}.
$$
For a parameter $\tau \in (0,1]$ we define
$$
\psi_\tau(t,x) = \eta_\tau(t) \phi_\tau(x), \quad \eta_\tau(t) = \eta\left( \frac{t}{\tau} \right) \quad \text{and} \quad \phi_{\tau}(x) = \phi\left( \frac{x}{\tau^{\frac{1}{2}}}\right).
$$

In the sequel we consider the following positive quantity:
$$
I(\tau) = \int_0^T \int_{\R^n} |u(t,x)|^{\alpha}  \psi_{\tau}(t,x) dxdt.
$$
In the following Lemma \ref{lemma_integral_inequality_test_function} we shall employ the well-known test function method to derive an integral inequality.

\begin{lemma}\label{lemma_integral_inequality_test_function}
Suppose that $u \in L^\alpha_{\text{loc}}([0, T) \times \R^n)$ is a weak solution to \eqref{Eq:MainProblem} with $u_0 \equiv 0$. Then, there exists $C = C(\alpha, n) > 0$ such that
	$$
	\left| \int_0^T \int_{\R^n} u(t,x) \{\partial^2_t + \Delta^2_x + 1\} \psi_{\tau} (t,x) dxdt \right| \leq C \tau^{-2 + \left( 1+\frac{n}{2} \right)\frac{1}{\alpha'}} I(\tau)^{\frac{1}{\alpha}},
	$$
where $\alpha'$ denotes the conjugate exponent of $\alpha$.
\end{lemma}
\begin{proof}
	From H\"older inequality we get
	\begin{align*}
		\left| \int_0^T \int_{\R^n} u(t,x) \partial^2_t \psi_{\tau} (t,x) dxdt \right| &\leq \tau^{-2} \int_0^T \int_{\R^n} |u(t,x)| \psi_\tau(t,x)^{\frac{1}{\alpha}} \left|\eta''\left( \frac{t}{\tau} \right)\right| \eta_\tau(t)^{-\frac{1}{\alpha}} \phi_{\tau}(x)^{\frac{1}{\alpha'}}  \phi_\tau(x) dxdt \\
		&\leq C_{\alpha} \tau^{-2} I(\tau)^{\frac{1}{\alpha}} \left\{ \int_0^\tau \int_{|x| \leq \frac{\tau}{2} } dxdt \right\}^{\frac{1}{\alpha'}} \\
		&\leq C_{\alpha, n} \tau^{-2 + \left(1 + \frac{n}{2} \right)\frac{1}{\alpha'}} I(\tau)^{\frac{1}{\alpha}}.
	\end{align*}
	Analogously, we have
	$$
	\left| \int_0^T \int_{\R^n} u(t,x) \Delta^2_x \psi_{\tau} (t,x) dxdt \right| \leq C_{\alpha, n} \tau^{-2 + \left(1 + \frac{n}{2} \right)\frac{1}{\alpha'}}I(\tau)^{\frac{1}{\alpha}}.
	$$
	On the other hand, since $\tau \leq 1$ we get
	$$
	\left| \int_0^T \int_{\R^n} u(t,x) \psi_{\tau} (t,x) dxdt \right| \leq C_{\alpha, n} \tau^{\left(1 + \frac{n}{2} \right)\frac{1}{\alpha'}}I(\tau)^{\frac{1}{\alpha}}
	\leq C_{\alpha, n} \tau^{-2 + \left(1 + \frac{n}{2} \right)\frac{1}{\alpha'}}I(\tau)^{\frac{1}{\alpha}}.
	$$
	This completes the proof.
\end{proof}
Now let us consider the following function:
\begin{equation}\label{eq_counter_example_non_local_existence}
	f(x) =
	\begin{cases}
		|x|^{-k}, \quad |x| \leq 1, \\
		0, \qquad \,\,\,\,\, |x|>1,
	\end{cases}
\end{equation}
with $k < n$. Hence $f \in L^{1}(\R^n)$.

\begin{lemma}\label{lemma_inequality_initial_datum}
	Let $f=f(x)$ be defined as above. If $k < n$, then there exists $C = C(n, k) > 0$ such that
	$$
	\int_{\R^n} f(x) \phi_{\tau}(x) dx \geq C	\tau^{\frac{n}{2} - \frac{k}{2}}.
	$$
\end{lemma}
\begin{proof}
	We have
	$$
	\int_{\R^n} f(x) \phi \left( \frac{x}{\tau^{\frac{1}{2}}} \right) dx =
	\int_{|x| \leq \tau^{\frac{1}{2}}} |x|^{-k} \phi \left( \frac{x}{\tau^{\frac{1}{2}}} \right) dx =
	\tau^{\frac{n}{2} - \frac{k}{2} }\int_{|x| \leq 1} |x|^{-k} \phi (x) dx = C \tau^{\frac{n}{2} - \frac{k}{2} }.
	$$
\end{proof}

\begin{proof}[Proof of Theorem \ref{non-existence_local}.]
	Let $f=f(x)$ given by \eqref{eq_counter_example_non_local_existence}. In order to conclude that $f \in L^1(\R^n) \cap L^m(\R^n)$ we shall assume
	$$
	k < \frac{n}{m}.
	$$
	Suppose now that there exist a weak solution $u(t,x) \in L^{\alpha}_{\text{loc}}([0,T) \times \R^n)$ to \eqref{Eq:MainProblem} with $u_0 \equiv 0$ and $u_1(x) = \lambda f(x)$ for some $\lambda \geq 0$. Combining Lemma \ref{lemma_integral_inequality_test_function} with Lemma \ref{lemma_inequality_initial_datum} we get
	$$
	C\lambda \tau^{\frac{n}{2} - \frac{k}{2} } + I(\tau) \leq C \tau^{-2 + \left( 1+\frac{n}{2} \right)\frac{1}{\alpha'}}I(\tau)^{\frac{1}{\alpha}}.
	$$
	Young's inequality for products allow us to conclude
	$$
	\lambda \leq C \tau^{-2\alpha' + 1 +\frac{k}{2}}.
	$$
	Therefore, if
	$$
	-2\alpha' + 1 +\frac{k}{2} > 0 \iff k > 2 \frac{\alpha+1}{\alpha -1},
	$$
	then we conclude that $\lambda = 0$ letting $\tau \to 0^{+}$.
To have a non-empty range for $k$ we then need $n>2m$ and
	$$
		\frac{n}{2m} > \frac{\alpha+1}{\alpha -1} \iff \alpha > \frac{n+2m}{n-2m}.
	$$
\end{proof}

\section*{Appendix}
\addcontentsline{toc}{section}{Appendix}
\renewcommand{\thesection}{A}
\section{Auxiliary results}
\noindent

To handle with $L^p-L^q$ estimates we state the important Riesz-Thorin interpolation theorem.
\begin{theorem} [Riesz-Thorin interpolation theorem, see Theorem 24.5.1 in \cite{EbertReissigBook}] \label{Riesz-Thorin}
Let $1\leq p_0,p_1,q_0,q_1\leq\infty$. If $T$ is a linear continuous operator in the space $\mathcal{L}\big( L^{p_0}\mapsto L^{q_0} \big)\cap\mathcal{L}\big( L^{p_1}\mapsto L^{q_1} \big)$, then
\[ T\in\mathcal{L}\big( L^{p_\theta}\mapsto L^{q_\theta} \big) \]
for any $\theta\in(0,1)$, where
\[ \frac{1}{p_\theta}=\frac{1-\theta}{p_0}+\frac{\theta}{p_1} \quad \text{and} \quad \frac{1}{q_\theta}=\frac{1-\theta}{q_0}+\frac{\theta}{q_1}. \]
Moreover, the following norm estimates hold:
\[ \|T\|_{\mathcal{L}(L^{p_\theta}\mapsto L^{q_\theta})} \leq \|T\|_{\mathcal{L}(L^{p_0}\mapsto L^{q_0})}^{1-\theta}\|T\|_{\mathcal{L}(L^{p_1}\mapsto L^{q_1})}^\theta. \]
\end{theorem}

A key result for multipliers in $M_p$ with $p\in(1,\infty)$ is the Mikhlin-H\"{o}rmander multiplier theorem.
\begin{theorem} [Mikhlin-H\"{o}rmander multiplier theorem, see Theorem 2.5 in \cite{Hormander_1960_estimates}] \label{Sec:App_Thm_Mikhlin-Hormander}
Let $1<p<\infty$ and $k=[n/2]+1$. Assume that $m\in\mathcal{C}^k(\mathbb{R}^n\slash{0})$ and
\[ |\partial_\xi^\gamma m(\xi)| \leq C|\xi|^{-|\gamma|}, \qquad |\gamma|\leq k. \]
Then, $m\in M_p^p$.
\end{theorem}

The classical integer version of the Gagliardo–Nirenberg inequality can be stated as follows.
\begin{prop}[Classical Gagliardo-Nirenberg inequality, see Proposition 25.5.4 in \cite{EbertReissigBook}]  \label{GagliardoNirenberginequality}
Let $j,m\in\mathbb{N}$ with $j<m$, and let $u\in \mathcal{C}^{m}_c(\mathbb{R}^n)$, i.e. $u\in \mathcal{C}^m$ with compact support. Let $\theta\in [j/m,1]$, and let $p,q,r$ in $[1,\infty]$ be such that
\[ j-\frac{n}{q} = \Big(m-\frac{n}{r}\Big)\theta-\frac{n}{p}(1-\theta). \]
Then,
\begin{align} \label{Eq:Gagliardo-Nirenberg-App}
\|D^ju\|_{L^q} \leq C_{n,m,j,p,r,\theta} \|D^m u\|_{L^r}^\theta\ \|u\|_{L^p}^{1-\theta}
\end{align}
provided that $\big(m-\frac{n}{r}\big)-j \not\in\mathbb{N}$, that is $\frac{n}{r}>m-j$ or $\frac{n}{r}\not\in\mathbb{N}$.

If $ \big(m-\frac{n}{r}\big)-j \in\mathbb{N}$, then \eqref{Eq:Gagliardo-Nirenberg-App} holds provided that $\theta\in [\frac{j}{m},1)$.
\end{prop}
To handle with $L^1-L^\infty$ estimate, we have the following lemma.
\begin{theorem}[Bernstein's Theorem, see Theorem 1.2 in \cite{Sjostrand}] \label{Theorem:Bernstein}
Let $n\geq1$ and $N>n/2$. We assume that $m\in H^N$, then $\mathcal{F}^{-1}m\in L^1$ and there exists a constant $C$ such that
\[ \|\mathcal{F}^{-1}m\|_{L^1} \leq C\|m\|_{L^2}^{1-\frac{n}{2N}}\|D^Nm\|_{L^2}^{\frac{n}{2N}}. \]
\end{theorem}
The estimates provided by Theorem \ref{Theorem:Bernstein} will be used together with the estimates for $\|\mathcal{F}^{-1}m\|_{L^\infty}$, provided by the following application of Littman's lemma, based on stationary phase methods.
\begin{lemma} [A Littman type lemma, Lemma 2.1 in \cite{Pecher}] \label{LittmanLemma}
Let us consider for $\tau\geq \tau_0$, $\tau_0$ is a large positive number, the oscillatory integral
\[ \mathcal{F}^{-1}_{\xi\rightarrow x}\big( e^{-i\tau\omega(\xi)}v(\xi) \big). \]
The amplitude function $v=v(\xi)$ is supposed to belong to $\mathcal{C}_c^{\infty}(\mathbb{R}_\xi^n)$ with support $\{\xi\in\mathbb{R}^n: |\xi|\in [1/2,2]\}$. The function $\omega=\omega(\xi)$ is $\mathcal{C}^\infty$ in a neighborhood of the support of $v$. Moreover, the Hessian $H_\omega(\xi)$ is non-singular, that is, $\det H_\omega(\xi)\neq0$, on the support of $v$. Then, the following $L^\infty-L^\infty$ estimate holds:
\[ \big\| \mathcal{F}^{-1}_{\xi\rightarrow x}\big( e^{-i\tau\omega(\xi)}v(\xi) \big) \big\|_{L^\infty(\mathbb{R}_\xi^n)}\leq C(1+\tau)^{-\frac{n}{2}}\sum_{|\alpha|\leq n+1}\big\| D_\xi^\alpha v(\xi) \big\|_{L^\infty(\mathbb{R}_\xi^n)}, \]
where the constant $C$ depends on $\mathcal{C}^{n+2}$ norm of the phase function $\omega$, on the lower bound for $|\det H_\omega(\xi)|$ and on the diameter of the support of $v$.
\end{lemma}
In \cite[Proposition 2.5, Chapter 8]{SS}, one can find a simple proof of Lemma \ref{LittmanLemma}, from which it is easy to check that the statement remains valid whenever~$\omega$ and~$v$ depend on some parameter, with a constant $C$, uniform with respect to the parameter, provided that the derivatives of $\omega$ and $v$, as well as $|\det H_\omega(t,\xi)|$ can be estimated uniformly with respect to the parameter.

\addcontentsline{toc}{section}{Appendix}
\renewcommand{\thesection}{B}

\section{Proof of Lemma \ref{Lemma_optimality}}\label{optimality}

\noindent

This Appendix is devoted to discuss the proof of Lemma \ref{Lemma_optimality}. Before to proceed, we need to recall some facts about Bessel functions. As usual, for $\nu \geq -\frac{1}{2}$ we denote by $J_{\nu}$ the Bessel function
$$
J_\nu(z) = \left( \frac{z}{2} \right)^{\nu} \sum_{k = 0}^{\infty} (-1)^{k} \frac{\left( \frac{1}{4} z^2 \right)^{k}}{ k! \Gamma(\nu+k+1)}.
$$
The Bessel functions satisfy the following asymptotic behaviour:
\begin{equation}\label{eq_assymptotic_bessel}
J_\nu (z)=c_\nu z^{-\frac12} \cos(z-\nu \pi/2-\pi/4)+ O(z^{-\frac 32}), \quad \text { as } \, z \to \infty.
\end{equation}
When $f \in L^1(\R^n)$ is radial, its Fourier transform can be computed in the following way:
$$
\widehat{f}(\xi) = \int_{0}^{\infty} f(r) r^{n-1} (r|\xi|)^{\frac{2-n}{2}} J_{\frac{n-2}{2}}(r|\xi|)dr.
$$
We also need the following result concerning oscillatory integrals (see Lemma $4$ in \cite{MarshallStrauss1980}).

\begin{lemma}\label{lemma_asymptotics_oscillatory_integrals}
	Let $h: [\beta, \gamma] \to \R$ be a concave or convex function and we consider
	$$
	G = \int_{\beta}^{\gamma} e^{ih(y)} dy.
	$$
	If $|h'(y)| \geq \lambda$ on $[\beta, \gamma]$, then there exists $C > 0$ independent of $\lambda$ such that
	$$
	|G| \leq C \lambda^{-1}.
	$$
\end{lemma}

\begin{proof}[Proof of Lemma \ref{Lemma_optimality}]
    We recall that
    $$
    \widehat{K}(t,\xi) = \frac{\sin(t\sqrt{1+|\xi|^4})}{\sqrt{1+|\xi|^4}} = \frac{\sin(t \langle |\xi|^2 \rangle )}{\langle |\xi|^2 \rangle}, \quad \langle \xi \rangle := \sqrt{1+|\xi|^2}.
    $$
 	Next, we fix $f \in \mathcal{S}(\R^n)$ such that $\widehat f  \in \mathcal{C}_c^\infty(\R^n)$, $\widehat{f}(\xi)$ is radial and $\widehat{f}$ vanishes near the origin.
        From the representation of the Fourier transform in terms of Bessel functions we have
        $$
        K(t,x) \ast f(x) = |x|^{1-n/2}  \int_0^\infty \frac{\sin(t \langle r^2 \rangle)}{\langle r^2 \rangle} \widehat{f}(r) \,J_{(n-2)/2}(|x|r) \,r^{\frac{n}{2} }\, dr.
        $$

       	Let us fix $a, b$ such that $0 < a < b$. We shall restrict our analysis in the zone $x \in \R^n$ such that
      	$$
      	ta < |x| < tb.
      	$$
      	Hence, we have $|x|/t \approx 1$. From \eqref{eq_assymptotic_bessel} we get
       	\begin{align*}
       	K(t,x) \ast f(x) &= c_n |x|^{\frac{1}{2}-\frac{n}{2}} \int_0^\infty \frac{\sin(t \langle r^2 \rangle)}{\langle r^2 \rangle} \widehat{f}(r) \cos\left(|x|r - \frac{n-1}{4}\pi\right) r^{\frac{n}{2}-\frac{1}{2}} dr \\
       		 & \quad + |x|^{1-\frac{n}{2}} \int_0^\infty \frac{\sin(t \langle r^2 \rangle)}{\langle r^2 \rangle} \widehat{f}(r) H(|x|r) r^{\frac{n}{2}} dr,
       	\end{align*}
       where $H(|x|r) \leq C (|x|r)^{-\frac{3}{2}}$ for $|x|r$ large. Since $\widehat{f}$ has compact support away from zero, we obtain
       \begin{align*}
       	 K(t,x) \ast f(x) &= c_n |x|^{\frac{1}{2}-\frac{n}{2}} \int_0^\infty \frac{\sin(t \langle r^2 \rangle)}{\langle r^2 \rangle} \widehat{f}(r) \cos\left(|x|r - \frac{n-1}{4}\pi\right) r^{\frac{n}{2}-\frac{1}{2}} dr + G(|x|)|x|^{1-\frac{n}{2}},
       \end{align*}
       where $G(|x|)|x|^{1-\frac{n}{2}} \leq C |x|^{-\frac{1}{2}-\frac{n}{2}}$ for all large $|x|$.

      Since we are interested in $at \leq |x| \leq bt$, we conclude that $G(|x|)|x|^{1-\frac{n}{2}} \lesssim t^{- \frac{1}{2}-\frac{n}{2}}$. So, the estimate \eqref{eq_estimate_from_below} must come from
      $$
      \tilde{I} = |x|^{\frac{1}{2}-\frac{n}{2}} \int_0^\infty \frac{\sin(t \langle r^2 \rangle)}{\langle r^2 \rangle} \widehat{f}(r) \cos\left(|x|r - \frac{n-1}{4}\pi\right) r^{\frac{n}{2}-\frac{1}{2}} dr.
      $$
      Using the formula
      $$
      \sin(a+b) + \sin(a-b) = 2 \sin(a) \cos(b),
      $$
      we write
      $$
      \sin(t \langle r^2 \rangle) \cos\left(|x|r - \frac{n-1}{4}\pi\right) = \frac{1}{2} \sin \left( t \langle r^2 \rangle + |x|r - \frac{n-1}{4} \pi \right)
      + \frac{1}{2} \sin \left( t \langle r^2 \rangle - |x|r + \frac{n-1}{4} \pi \right).
      $$
      Hence, we have
      $$
      \tilde{I} = \tilde{c}_n |x|^{\frac{1}{2}-\frac{n}{2}} \{ I_{+} + I_{-}\},
      $$
      where
      $$
      I_{\pm} = \int_{0}^{\infty} \sin\left( t\langle r^2 \rangle \pm \left(|x|r - \frac{n-1}{4}\pi \right) \right) \frac{\widehat{f}(r) r^{\frac{n}{2} -\frac{1}{2}}}{\langle r^2 \rangle} dr.
      $$
      Setting $g(r) = \frac{\widehat{f}(r) r^{\frac{n}{2} -\frac{1}{2}}}{\langle r^2 \rangle}$, we have
      $$
      I_{\pm}(t) = \int_{0}^{\infty} (2i)^{-1}\{ e^{ih_{\pm}(t,r)} - e^{-ih_{\pm}(t,r) }\} g(r) dr,
      $$
      where
      $$
      h_{\pm}(t,r) = t\langle r^2 \rangle \pm \{ |x|r - \frac{n-1}{4}\pi \}.
      $$

      Observe that
      $$
      \partial_r h_{\pm}(t,r) = \frac{2tr^3}{\langle r^2 \rangle} \pm |x|, \quad \partial^2_r h_{\pm}(t,r) = \frac{2tr^2}{\langle r^2 \rangle^{3}} \{3+r^4\}.
      $$
      Since $g$ has compact support away from zero and $at <|x| < bt$, we have
      $$
      \partial^2_r h_{\pm} > 0, \quad \partial_r h_{+}(t,r) \gtrsim t.
      $$
      Therefore, Lemma \ref{lemma_asymptotics_oscillatory_integrals} implies
      $$
      |I_+| \lesssim t^{-1} \implies |x|^{\frac{1}{2}-\frac{n}{2}} |I_+| \lesssim t^{-\frac{1}{2} - \frac{n}{2}}.
      $$
  	  Hence, we must get \eqref{eq_estimate_from_below} from $|x|^{\frac{1}{2}-\frac{n}{2}}I_{-}$. Let us write
  	  $$
  	  h_{-}(t,r) = t h(r), \quad h(r) := \langle r^2 \rangle - \frac{|x|}{t}r + \frac{(n-1)\pi}{4t}.
  	  $$
  	  Since $\partial^{2}_r h(r) > 0$ we obtain that $\partial_r h(\cdot)$ has a single zero $r_0$ which satisfies the equation
  	  \begin{equation}\label{eq_stationary_point}
  	  \frac{2r_0^3}{\langle r_0^2 \rangle} = \frac{|x|}{t}.
  	  \end{equation}
  	  In particular, we have
  	  \begin{equation}\label{eq_lower_bound_r_0}
  	  \frac{|x|}{t} = \frac{2r_0^3}{\langle r_0^2 \rangle}  \leq 2r_0 \implies r_0 \geq \frac{a}{2}.
  	  \end{equation}
  	  On the other hand, (using $r_0 \geq \frac{a}{2}$)
  	  \begin{equation}\label{eq_upper_bound_r_0}
  	  \frac{|x|}{t} = \frac{2r_0^3}{\langle r_0^2 \rangle} \geq 2r_0 \frac{r^2_0}{1+r_0^2} \geq 2r_0 \frac{ \frac{a^2}{4} }{ 1+\frac{a^2}{4} } \implies r_0 \leq \frac{b}{2} \frac{4+a^2}{a^2}.
  	  \end{equation}
  	  That is, since $\frac{|x|}{t} \approx 1$, we conclude that $r_0$ is bounded away from zero and infinity.
  	  In this way, we can choose $f$ such that $r_0 \in \supp\widehat{f}$ and we can write $h'(r)=(r-r_0)h_0(r)$, with $h_0(r_0)\neq 0$. Now we split our analysis in two cases: $|r-r_0| \leq t^{-\frac{1}{3}}$ or $|r-r_0| \geq t^{-\frac{1}{3}}$ ($t > 1$).
  	
	  If $ |r-r_0| \geq t^{-\frac{1}{3}}$ then $|th'(r)| \geq c t^{\frac{2}{3}}$. 
  	  From Lemma \ref{lemma_asymptotics_oscillatory_integrals} we conclude
  	  $$
  	  |x|^{\frac{1}{2}-\frac{n}{2}} \left| \int_{|r-r_0| \geq t^{-\frac{1}{3}}} \sin(th(r)) g(r) dr \right| \leq C t^{\left(\frac{1}{2}-\frac{2}{3}\right) - \frac{n}{2}}.
  	  $$
  	  Therefore, to conclude the proof it suffices to prove that
  	  $$
  	  \int_{|r-r_0| \leq t^{-\frac{1}{3}}} \sin(th(r)) g(r) dr \gtrsim t^{-\frac{1}{2}}.
  	  $$
  	
  	  We write Taylor's expansion up to the order three to $h(r)$
  	  $$
  	  T_3(r) = h(r_0) + \frac{h''(r_0)}{2}(r-r_0)^{2} + \frac{h^{(3)}(r_0) }{6}(r-r_0)^{3} = h(r_0) + \alpha (r-r_0)^2 + \beta (r-r_0)^3,
  	  $$
  	  where $\alpha > 0$. Thus, we have
  	  \begin{align*}
  	  	\int_{|r-r_0| \leq t^{-\frac{1}{3}}} \sin(th(r)) g(r) dr &= \int_{|r-r_0| \leq t^{-\frac{1}{3}}} \big( \sin(th(r)) - \sin(tT_3(r)) \big)g(r) dr + \int_{|r-r_0|\leq t^{-\frac{1}{3}}} \sin(T_3(r))g(r)dr \\
  	  	&= \int_{|r-r_0|\leq t^{-\frac{1}{3}}} \sin(T_3(r))g(r)dr + O(t^{-\frac{2}{3}}).
  	  \end{align*}
  	  Indeed, using the Lagrange remainder for some $\overline{r} \in (r_0 - t^{-\frac{1}{3}}, r_0 + t^{-\frac{1}{3}})$ we get
  	  $$
  	  \big| \sin(th(r)) - \sin(tT_3(r)) \big| \leq t|h(r) - T_3(r)| \lesssim t \frac{|h^{(4)}(\overline{r})|}{4!}(r-r_0)^{4} \lesssim t |r-r_0|^{4}.
  	  $$
  	  Hence, we have
  	  $$
  	  \left| \int_{|r-r_0|\leq t^{-\frac{1}{3}}}  \big( \sin(th(r)) - \sin(tT_3(r)) \big) g(r) dr\right| \lesssim t\int_{|r-r_0|\leq t^{-\frac{1}{3}}} (r-r_0)^{4} dr \lesssim t^{-\frac{2}{3}}.
  	  $$
  	  Applying a similar idea to the function $g(r)$, we conclude that it suffices to prove
  	  $$
  	  g(r_0)\int_{|r-r_0|\leq t^{-\frac{1}{3}}} \sin(T_3(r))dr = g(r_0)\int_{-t^{-\frac{1}{3}}}^{t^{-\frac{1}{3}}} \sin\big( th(r_0) + t\alpha s^2 + t \beta s^3 \big)ds \gtrsim t^{-\frac{1}{2}}.
  	  $$
  	
  	  Since $\alpha > 0$ we obtain that the function
  	  $
  	  F(s) = s\sqrt{\alpha+\beta s}
  	  $
  	  is well-defined for small values of $s$. Moreover, again assuming $s$ small, we have
  	  $$
  	  F'(s) = \sqrt{\alpha+\beta s} - \frac{s\beta}{2\sqrt{\alpha+\beta s}} \gtrsim \sqrt{\alpha} > 0.
  	  $$
  	  So, $F$ defines a diffeomorphism near the origin: $F : (-\delta, \delta) \to (F(-\delta), F(\delta))$. Changing the variables we get
	  \begin{align*}
	  \int_{-t^{-\frac{1}{3}}}^{t^{-\frac{1}{3}}} \sin\big( th(r_0) + t\alpha s^2 + t \beta s^3 \big) ds &= \int_{F(-t^{-\frac{1}{3}})}^{F(t^{-\frac{1}{3}})} \sin\big( th(r_0) + ts^2 \big) |(F^{-1})'(s)| ds \\ &\geq \int_{-d t^{-\frac{1}{3}}}^{dt^{-\frac{1}{3}}} \sin\big( th(r_0) + ts^2 \big) |(F^{-1})'(s)| ds,
	  \end{align*}
	  where $d = (2^{-1}\alpha)^{\frac{1}{2}}$. Applying Taylor's formula to $|(F^{-1})'(s)|$ we see that it is enough to prove that
	  \begin{align}\label{eq_}
	  \int_{-d t^{-\frac{1}{3}}}^{dt^{-\frac{1}{3}}} \sin\big( th(r_0) + ts^2 \big) ds \gtrsim t^{-\frac{1}{2}}.
	  \end{align}
	
	  Using Lemma \ref{lemma_asymptotics_oscillatory_integrals} once more, the integral \eqref{eq_} may be considered over the whole real line with an error of order $O(t^{-\frac{2}{3}})$. Now we note that
	  \begin{align*}
	  \int_{-\infty}^{\infty}  \sin\big( th(r_0) + ts^2 \big) ds &= t^{-\frac{1}{2}} \int_{-\infty}^{\infty}  \sin\big( th(r_0) + s^2 \big) ds \\
	  &= t^{-\frac{1}{2}} (2i)^{-1} \int_{-\infty}^{\infty} e^{i\left(th(r_0)+s^2\right)} - e^{-i\left( th(r_0)+s^2 \right)} ds  \\
	  &= t^{-\frac{1}{2}} \sqrt{\frac{\pi}{2}} (2i)^{-1} \big\{ e^{ith(r_0)}(1+i)  - e^{-ith(r_0)} (1-i) \big\}  \\
	  &= t^{-\frac{1}{2}} \sqrt{\frac{\pi}{2}} \big\{ \cos(th(r_0)) + \sin(th(r_0)) \big\} \\
	  &= t^{-\frac{1}{2}} \sqrt{\pi} \sin\left(th(r_0) + \frac{\pi}{4}\right).
	  \end{align*}
	  Moreover
	  $$
	  th(r_0) + \frac{\pi}{4} = t\left( \langle r_0^2 \rangle - \frac{|x|}{t}r_0 + \frac{(n-1)\pi}{4t} \right) + \frac{\pi}{4}.
	  $$
	 Since $r_0$ satisfies \eqref{eq_stationary_point}, \eqref{eq_lower_bound_r_0} and \eqref{eq_upper_bound_r_0},
we have
\[  \langle r_0^2 \rangle - \frac{|x|}{t}r_0 =\frac{|x|}{t}r_0\left( 2r_0^2 \frac{t^2}{|x|^2}-1\right)\geq ar_0\left( \frac{a^2}{2b^2}-1\right).
\]
	  $$
	  \langle r_0^2 \rangle - \frac{|x|}{t}r_0 = \langle r^2_0 \rangle \left( 1 - \frac{|x|^2}{t^2} \frac{1}{2r^2_0} \right),
	  $$
	  and
	  $$
	  1- \frac{2b^2}{a^2} \leq  1 - \frac{|x|^2}{t^2} \frac{1}{2r^2_0} \leq 1 - \frac{2a^2}{b^2} \frac{a^4}{(4+a^2)^2}.
	  $$
	  Hence, choosing $0 < a < b$ such that $b \leq \frac{{\sqrt 2}a^3}{4+a^2}$ we have that $\langle r_0^2 \rangle - \frac{|x|}{t}r_0$ is bounded away from zero for all $x$ and $t$ such that $at \leq |x| \leq tb$. Hence, it is possible to take a sequence $1 \leq t_k \to \infty$ such that
	  $$
	  \sin\left(t_kh(r_0) + \frac{\pi}{4}\right) \geq \frac{1}{2}, \quad k \in \N.
	  $$
	  Therefore, we get
	  $$
	  \int_{-\infty}^{\infty}  \sin\big( t_kh(r_0) + t_ks^2 \big) ds \gtrsim t_k^{-\frac{1}{2}}.
	  $$
	  This concludes the proof.
    \end{proof}

\section*{Acknowledgments}
Alexandre A. Junior has been supported by Grant $2022/01712-3$ and Halit S. Aslan has been supported by Grant $2021/01743-3$ from  S\~ao Paulo Research Foundation (FAPESP). 
Marcelo R. Ebert  is partially supported by "Conselho Nacional de Desenvolvimento Cient\'ifico e
Tecnol\'ogico (CNPq)",  grant number 304408/2020-4.

\end{document}